\documentclass[a4paper,10pt,reqno]{amsart}

\usepackage[english]{babel}
\usepackage[latin1]{inputenc}
\usepackage{amsmath,amssymb,amsfonts,amsthm,amscd}
\usepackage[hidelinks]{hyperref}[urlcolor=blue]
\usepackage[usenames,dvipsnames,svgnames,table]{xcolor}
\usepackage[mathscr]{eucal}
\usepackage{latexsym}
\usepackage{multicol}
\usepackage{color}
\usepackage{multirow}
\usepackage{enumitem}
\usepackage[a4paper]{geometry}
\usepackage{subcaption} 
\usepackage{graphicx}

\definecolor{alertmanzano}{rgb}{0.8,0,0.3}

\newcommand{\alertm}[1]{%
	\marginpar{%
		\ifodd\value{page} \raggedright \else \raggedleft \fi
		\footnotesize{\textcolor{alertmanzano}{#1}}
	}
}

\newcommand{\N}{\mathbb{N}}

\newcommand{\R}{\mathbb{R}}

\newcommand{\M}{\mathbb{M}}
\newcommand{\h}{\mathbb{H}}

\newcommand{\E}{\mathbb{E}}

\newcommand{\dist}{\text{dist}}

\newcommand{\arcsinh}{\mathop{\rm arcsinh}\nolimits}

\newcommand{\Nil}{\mathrm{Nil}_3}
\newcommand{\PSL}{\widetilde{\mathrm{SL}}_2(\mathbb{R})}
\renewcommand{\div}{\mathop{\rm div}\nolimits}
\DeclareMathOperator{\arctanh}{arctanh}

\newtheorem{theorem}{Theorem}[section]
\newtheorem*{theorem*}{Theorem}
\newtheorem{proposition}[theorem]{Proposition}

\newtheorem{lemma}[theorem]{Lemma}

\theoremstyle{definition}

\theoremstyle{remark}
  \newtheorem{remark}{Remark}

\numberwithin{equation}{section}

\title{ Genus one $H$-surfaces with $k$-ends in $\h^2\times\R$}
\date{}

\author{Jes\'us Castro-Infantes}
\address{Departamento de Matemática Aplicada a las TIC, Universidad Politécnica de Madrid}
\email{jesus.castro@upm.es}

\author{Jos\'e S. Santiago}
\address{Departamento de Matem\'aticas, Universidad de Ja\'en, 23071 Ja\'en SPAIN}
\email{jssantia@ujaen.es}

\begin{document}
\maketitle
\begin{abstract}
We construct two different families of properly Alexandrov-immersed surfaces in $\mathbb H^2\times \mathbb R$ with constant mean curvature $0<H\leq \frac 1 2$, genus one and $k\geq2$ ends ($k=2$ only for one of these families). These ends are asymptotic to vertical $H$-cylinders for $0<H<\frac 1 2$. This shows that there is not a Schoen-type theorem for immersed surfaces with positive constant mean curvature in $\mathbb H^2\times\mathbb R$. These surfaces are obtained by means of a conjugate construction.
\end{abstract}
\section{Introduction}

In 1983, R. Schoen~\cite{Sc} proved that the unique complete immersed minimal surfaces in $\R^3$ with  finite total curvature and two embedded ends are the catenoids. Concerning surfaces with  constant mean curvature $H>0$ ($H$-surfaces in the sequel), Koreevar, Kusner and Solomon~\cite{KKS} proved that any complete properly embedded  $H$-surface in $\R^3$ with two ends is a rotationally invariant Delaunay surface. If we drop the hypothesis of being properly embedded, we have that Kapouleas~\cite{Kap} constructed immersed  $H$-surfaces in $\R^3$ with two ends and genus $g\geq 2$.

 Korevaar, Kusner, Meeks and Solomon~\cite{KKMS} proved analogous results in the hyperbolic space $\h^3$ showing that the only properly embedded  $H$-surfaces in $\h^3$ with two ends and  $H>1$ are the hyperbolic Delaunay surfaces.  In $\h^3$, $H$-surfaces with $H=1$ are known as Bryant surfaces and the value $H=1$  is known as \emph{critical} in the literature since surfaces with subcritical, critical, supercritical mean curvature usually have different geometric features. Levitt and Rosenberg~\cite{LR} proved that a complete Bryant surface in $\h^3$ with asymptotic boundary consisting of at most two points must be a surface of revolution. Again, if we remove the hypothesis of being properly embedded, Rossman and Sato~\cite{RS} have constructed properly immersed Bryant surfaces with genus one and two ends, each of them asymptotic to a point in the ideal boundary of $\h^3$.
 
 In the product space $\h^2\times\R$, Hauswirth, Nelli, Sa Earp and Toubiana proved~in \cite{HNST} a Schoen-type theorem, showing that the horizontal catenoids are the unique properly immersed minimal surfaces with finite total curvature and two embedded ends, each of them asymptotic to a vertical plane. Later,  Hauswirth, Menezes and Rodríguez \cite{HMR} removed the hypothesis of having finite total curvature. Manzano and Torralbo~\cite{MT} showed that there are not properly immersed surfaces with $0\leq H\leq \tfrac 1 2$ at bounded distance from a horizontal geodesic. For supercritical $H$-surfaces in $\h^2\times\R$, that is,  $H>\frac 1 2$, Mazet~\cite{Mazet-15} proved that a properly embedded $H$-surface with $H>\frac 1 2$, finite topology and cylindrically bounded (with respect to a vertical geodesic) must be a rotational Delaunay surface.  
 In this article, we prove the following result concerning the subcritical  and critical case ($0<H\leq \frac 1 2)$, showing that there is not a Schoen-type theorem for immersed $H$-surfaces in $\h^2\times\R$.
 \begin{theorem}
 	There exists a family of properly immersed genus-one $H$-surfaces with $0<H\leq\frac{1}{2}$ and two ends. If $H<\frac{1}{2}$, each of these ends is asymptotic to a vertical $H$-cylinder from the convex side.
 \end{theorem}

It seems reasonable that, if we replace the property of being properly immersed by properly embedded, the unique complete $H$-surfaces in $\h^2\times\R$ with two ends asymptotic to a vertical $H$-cylinder ($H<\frac 1 2$) should be the $H$-catenoids and the embedded $H$-catenodoids constructed in~\cite{Plehnert2,CMR}. 

Our genus-one $H$-surfaces with two ends belong to a larger family of examples. In fact, we construct two different families of highly symmetric properly Alexandrov-embedded surfaces with genus one. The first family is called $(H, k)$-noids with genus one and they have  $k\geq 3$ ends,  each of them asymptotic to a vertical $H$-cylinder from the concave side (only for $H<\frac 1 2$), see Theorem~\ref{th:k-noides}. Moreover, we prove that the $(H, k)$-noids  are embedded for $H>\frac{1}{2}\cos(\frac{\pi}{k})$, see Proposition~\ref{Theorem:embebimiento}. 
 The second family is called $(H, k)$-nodoids with genus one and they have $k\geq 2$ ends, each of them asymptotic to a vertical $H$-cylinder from the convex side (only for $H<\frac 1 2$), see Theorem~\ref{th-k-nodoids}. These  examples are the first $H$-surfaces with non-zero genus  for $0<H<\frac 1 2$ in $\h^2\times\R$.

 We also construct the analogous $\frac 1 2$-surfaces in the critical case, recovering the $\frac 1 2$-surfaces with genus one and $k\geq 3$ ends constructed by Plehnert~\cite{Plehnert}; we are able to  show that all of them are embedded, see Theorem~\ref{th:k-noides} and Proposition~\ref{Theorem:embebimiento}.  Moreover, we construct a new family of non-embedded $\frac 1 2$-surfaces with genus one and $k\geq 2$ ends, see Theorem~\ref{th-k-nodoids}.
 
Furthermore, we can adapt our construction to produce $H$-surfaces in $\h^2\times\R$ with $0<H\leq\frac{1}{2}$ and infinitely many ends, that are invariant by a discrete group of parabolic or hyperbolic translations with one limit end and two limit ends respectively. We can ensure the embeddedness of some of these surfaces. This extends the results of Rodr\'iguez~\cite{R}  and Mart\'in and Rodr\'iguez~\cite{MR} for minimal surfaces but only for the case of infinitely many ends  and two limit ends. 

For $0<H<\frac 1 2$, our $H$-surfaces  are expected to have finite total curvature, because they are proper, have finite topology and the ends are asymptotic to vertical  $H$-cylinders. In the minimal case, these conditions are equivalent to have finite total curvature, see~\cite{HMR}. Even  in the minimal case, there are only a few examples  with finite total curvature in the literature: the vertical planes, the Scherk graphs constructed by Nelli and Rosenberg~\cite{NR} and Collin and Rosenberg~\cite{CR}, the Twisted Scherk graphs by Pyo and Rodríguez~\cite{PR}, the minimal $k$-noids by Pyo~\cite{Pyo} and Morabito and Rodríguez~\cite{MorRod}, the genus $g\geq 1$ minimal $k$-noids with large $k$ by Martín, Mazzeo and Rodríguez~\cite{MMR} and~the genus one minimal $k$-noids with $k\geq 3$ ends by the first author and Manzano~\cite{CM}.

We use a conjugate Jenkins-Serrin technique to construct our genus-one $H$-surfaces in $\h^2\times\R$, that are the counterparts of the genus-one minimal $k$-noids of~\cite{CM}. This technique is based in  Daniel's sister correspondence~\cite{Dan} between minimal surfaces in the homogeneous simply connected $3$-manifold $\mathbb E(4H^2-1,H)$ and $H$-surfaces in $\h^2\times\R$. It consists in solving a Jenkins-Serrin problem (a Dirichlet problem for the minimal surface equation with possible asymptotic boundary values $\pm\infty$ over geodesics) in $\mathbb E(4H^2-1,H)$ and, then, extend the conjugate $H$-surface by mirror symmetries until obtaining a complete $H$-surface with the desired topology. This technique has been a very fruitful tool in the construction of $H$-surfaces in product spaces, see~\cite{CMT} and the references therein. In particular, our construction uses ideas developed by Mazet~\cite{Maz}, where he constructs minimal $k$-noids ($k\geq 3$) with genus $1$ in $\R^3$, by  Plehnert~\cite{Plehnert}, where she constructs critical $k$-noids ($k\geq 3$) with genus $1$ in $\h^2\times\R$ and~\cite{CM}, where  the aforementioned minimal $k$-noids ($k\geq 3$) with genus $1$ in $\h^2\times\R$ are obtained.

The paper is organized as follows: in Section~\ref{sec:preliminares}, we introduce some tools that we will use throughout the rest of the paper. We extend the Generalized Maximum Principle proved in~\cite{CR} to the case of $\PSL$, and we give a brief introduction to the conjugate technique used in the construction of the $H$-surfaces. In Section~\ref{sec:construction}, we construct the genus one $(H,k)$-noids  and the genus one $(H,k)$-nodoids  for $0<H\leq\frac{1}{2}$.
\section{Preliminaries}\label{sec:preliminares}

Given $\kappa,\tau\in\R$, we denote by $\mathbb E(\kappa,\tau)$ the unique complete and simply-connected $3$-manifold that admits a Killing submersion $\pi:\E(\kappa,\tau)\to\mathbb M^2(\kappa)$ with constant bundle curvature $\tau$, being $\mathbb M^2(\kappa)$ the simply connected surface with constant curvature $\kappa$ and whose fibers are the integral curves of a unitary Killing vector field $\xi$, see for instance~\cite{Man} or~\cite[Section~2]{CMT}. We will use the so-called cylinder model
\begin{equation}\label{eq:model}
	\mathbb E(\kappa,\tau)=\{(x,y,z)\in\R^3: 1+\tfrac{\kappa}{4}(x^2+y^2)>0\}
\end{equation}
endowed with the Riemannian metric $ds^2=\lambda^2(dx^2+dy^2)+(dz+\lambda \tau(ydx-xdy))^2$, where $\lambda=(1+\frac{\kappa}{4}(x^2+y^2))^{-1}$ is the conformal factor of the metric in $\M^2(\kappa)$. This base surface is identified with a disk of radius $\frac{2}{\sqrt{-\kappa}}$ for $\kappa<0$ and with $\R^2$ for $\kappa\geq 0$. We choose the orientation such that 
\begin{align}\label{eq:frame}
	E_1=\frac{\partial_x}{\lambda}-\tau y\partial_z,\ \ \ E_2=\frac{\partial_y}{\lambda}+\tau x\partial_z\ \ \ \text{and}\ \ \ E_3=\partial_z
\end{align}

\noindent define a global positively oriented orthonormal frame. In this model, the Killing submersion over $\M^2(\kappa)$  reads as $\pi(x,y,z)=(x,y)$ and the unit Killing vector field is $\xi=\partial_z$.

In this article, we will focus in the case $\kappa\leq 0$, where the model~\eqref{eq:model} is global. Moreover, for $\kappa<0$, the space $\mathbb{E}(\kappa,\tau)$ corresponds to $\PSL$ ($\tau\neq 0$) or $\h^2(\kappa)\times\R$ ($\tau=0$). These spaces have a well defined  asymptotic boundary, see for instance~\cite{Cas}; more precisely, we identify topologically $\PSL$ and $\h^2\times\R$ and  consider the product compactification for $\h^2\times\R$. Then, the asymptotic boundary of $\PSL$, denoted by $\partial_\infty\PSL$, is homeomorphic to the vertical asymptotic boundary $\partial_\infty \h^2(\kappa)\times\R$ joint with the horizontal asymptotic boundaries $\h^2\times\{\pm \infty\}$. In this setting, we will say that a point $p\in\partial_\infty\PSL$ belongs to the asymptotic boundary of a surface $\Sigma\subset\PSL$ if there exists a divergent sequence $\{p_n\}$ in $\Sigma$ that converges to $p$ in the product compactification. Eventually, in the case of $\Nil$ ($\kappa=0$), we will refer to the ideal horizontal boundaries as $\R^2\times\{\pm\infty\}$; in this case, the ideal vertical boundary is not well defined.  
\subsection{Minimal graphs in $\mathbb E (\kappa,\tau)$}

A vertical graph is a section of the submersion $\pi:\E(\kappa,\tau)\to \mathbb M^2(\kappa)$ defined over a domain $U$ in $\mathbb M^2(\kappa)$. If we consider the model~\eqref{eq:model} and the zero section $F_0(x,y)=(x,y,0)$, we can parameterize this section in terms of a function $u:U\to \R$ as 
\begin{equation}
	F_u(x,y)=(x,y,u(x,y)),\ \ (x,y)\in U.
\end{equation} 
If $u\in C^2(U)$, the mean curvature of this vertical graph is computed as
\begin{equation}
	2H=\div\left(   \frac{Gu}{\sqrt{1+|Gu|^2}} \right),
\end{equation}
where $\div(\cdot)$ and $|\cdot|$ are the divergence and the norm in $\mathbb M^2(\kappa)$, respectively, and $Gu$ is the generalized gradient (see also~\cite{D-PMN}) given by
\[Gu=(u_x\lambda^{-2}+\tau y \lambda^{-1})\partial_x+(u_y\lambda^{-2}-\tau x \lambda^{-1})\partial_y.\]
Let $\Sigma$ be the minimal graph of $u:U\subset\M^2(\kappa)\to\R $. We define the \emph{flux along an arc $c\subset U$} as
\begin{equation}\label{eq:flux}
	\mathcal F(\Sigma,c)=\int_c \left\langle \frac{Gu}{\sqrt{1+|Gu|^2}},-J_{\mathbb M^2} \frac{c'}{|c'|}\right\rangle_{\mathbb M^2},
\end{equation} 
where $J_{\mathbb M^2}$ represents the counter clock-wise rotation of angle $\frac{\pi}{2}$ in  $T\mathbb M^2(\kappa)$ and therefore $-J_{\mathbb M^2} \frac{c'}{|c'|}$ is an unitary normal vector to $c$. Observe that, by the Divergence Theorem, if $c$ is a simple and closed curve such that $F_u(c)$ encloses a minimal disk in $\Sigma$, then $\mathcal F(\Sigma,c)=0$, see also~\cite{MRR,Younes,Me}.   If $\Sigma$ has boundary, $c$ is a convex arc of this boundary and $u$ is continuous on $c$, we can also define the flux in $c\subset\partial U$, see~\cite[Lemma~6.2]{Younes}, in that case $-J_{\mathbb M^2} \frac{c'}{|c'|}$ is a unitary conormal. Moreover if $c$ is a geodesic  arc of $\mathbb M^2$ and $u$ takes the asymptotic values $\pm \infty$ over $c$, then $\mathcal F(\Sigma,c)=\pm|c|$ depending if $-J_{\mathbb M^2} \frac{c'}{|c'|}$ is an inward or an outer conormal vector, see \cite[Lemma~6.3]{Younes}.

\begin{proposition}\label{prop:Flux}
	Let $\Sigma$ be a minimal graph with boundary in $\mathbb E(\kappa,\tau)$ and let $\gamma\subset\partial\Sigma$ be a curve parameterized by arc-length, then
	\begin{equation}\label{eq:Flujos}
		\mathcal F(\Sigma,\pi(\gamma))=\int_\gamma \langle -J\gamma',\xi\rangle,
	\end{equation}
	\noindent where $J$ is the rotation of angle $\frac{\pi}{2}$  in $T\Sigma$, such that $\{\gamma',J\gamma',N\}$ is a positively oriented orthonormal frame.
\end{proposition}
\begin{proof}

Let  $u:U\to\mathbb \R$ be the function defining the vertical graph  $\Sigma$.
 Assume that $\gamma:[a,b]\to\Sigma$, $\gamma(s)=(x(s),y(s),u(x(s),y(s)))$ is parameterized by arc-length. The upward normal along $\gamma$ is
	\[N=\frac{1}{\alpha}\left( -( \tau\lambda y+u_x )E_1+(\tau\lambda x -u_y)E_2 +\lambda \xi\right) ,\]
	where $\alpha^2=\lambda^2+ ( \tau\lambda y+u_x )^2+ (\tau\lambda x -u_y)^2$. By a straightforward computation, we get that
	\[\langle -J\gamma',\xi\rangle=\langle N\wedge \gamma',\xi\rangle=\frac {\lambda}{\alpha }\left( x'(\tau\lambda x -u_y)+y'(\tau\lambda y+u_x )\right).\]
	
	\noindent On the other hand, denoting $c=\pi(\gamma)=(x,y)$  and $X_u=\frac{Gu}{\sqrt{1+|Gu|^2}}$, we easily obtain that
	\begin{align*}
		\left\langle X_u,-J_{\mathbb M^2} \frac{c'}{|c'|}\right\rangle_{\mathbb M^2}&=
		\frac{1}{|c'|}\langle X_u,-y'\partial x+x'\partial_y\rangle_{\mathbb M^2}\\
		&=\frac{\lambda}{|c'|\alpha}\left( x'(\tau\lambda x -u_y)+y'(\tau\lambda y+u_x )\right).
	\end{align*}  
	\noindent Therefore we get the desired equation~\eqref{eq:Flujos}.
\end{proof}
 We have that Proposition~\ref{prop:Flux} holds true when $\gamma\subset\partial\Sigma$ and $u$ is continuous in $\gamma$ or when $\gamma$ is a horizontal geodesic in $\partial_\infty\Sigma$ (we are identifying $-J\gamma'$ with $\pm\xi$ in the limit when we approach the asymptotic boundary and the sign depends on the orientation and the asymptotic value that we take). By Proposition 2.1., we define  $\mathcal F(\Sigma,\gamma)=\mathcal{F}(\Sigma, \pi(\gamma))=\int_\gamma \langle -J\gamma',\xi\rangle$ for a curve $\gamma\subset\partial\Sigma$. It allows us to define the flux along a curve for minimal surfaces in general and not only for minimal graphs since Equation~\eqref{eq:flux} only depends on the conormal of the curve. Moreover, observe that, if the angle function $\nu$ vanishes along  $\gamma$, then  $\mathcal F(\Sigma,\gamma)=\pm|\pi(\gamma)|$.

Assume that $\kappa<0$ and  let $\gamma_1$ and $\gamma_2$ be two convex embedded arcs in $\h^2(\kappa)$ with vertex on the same ideal point $q_0\in \partial_\infty \h^2(\kappa)$. We say these arcs are \emph{asymptotic} at $q_0$ if $\dist(q,\gamma_i) \to 0$ as $q\to q_0$ with $q\in\gamma_j$ and $j\neq i$.  We will show in the next proposition that the Generalized Maximum Principle~\cite[Theorem 2]{CR} easily extends to $\PSL$. 
\begin{proposition}[Generalized Maximum Principle]\label{prop:G-M}
	Let $\Omega\subset \h^2(\kappa)$ be an unbounded piecewise regular domain such that $\partial \Omega \cap \partial_\infty \h^2(\kappa)$ is  finite and  every $q\in \partial\Omega \cap \partial_\infty \h^2(\kappa)$ is the endpoint of exactly two asymptotic arcs in $\partial \Omega$. Let $U\subset \Omega$ be a domain and $u,v\in C^0(\overline{U})$ functions that define  minimal graphs over $U$, $\Sigma_u$ and $\Sigma_v$,  respectively. If $\Sigma_u$ is below $\Sigma_v$ on $\partial U$, i.e., $u\leq v$ on $\partial U$, then $\Sigma_u$ is below $\Sigma_v$ on $U$.
\end{proposition}
\begin{proof}
	
	We will argue by contradiction. Let us suppose the set $A=\{p\in U:\  u(p)>v(p)\}$ is not empty. By the maximum principle for compact domains proved in~\cite{Younes}, we have that $\overline{A}$ cannot be compact. Without loss of generality, we consider that $A$ is a connected component since we can reason in the same way in each connected component. So we know that $\partial A$ contains arcs going into ideal points of $\Omega$, i.e., $\partial A$ has $n\geq 1$ ideal points. By hypothesis, each ideal point $q_i$, $i=1,\dots, n$, is the endpoint of two asymptotic arcs. We choose small horocycles $\mathcal{H}_i$ containing $q_i$ which  intersects once each asymptotic arc with vertex $q_i$. Let $h_i=\mathcal H_i\cap A$ and let $c_i$ be the compact arcs in $\partial A$ joining the curves $h_i$ with $h_{i+1}$ (we are assuming that $h_1=h_{n+1}$), see Figure~\ref{Fig-Collin_Krust}. We call $\alpha$ the continuous and piecewise smooth closed cycle composed by $\bigcup_i h_i\cup c_i$. By the Divergence Theorem, the flux of the function $u$ and $v$    over the cycle $\alpha$ is $0$. By Proposition~\ref{prop:Flux},   the flux of $u$ and the flux of $v$ in each curve $h_i$ is bounded by the length of $h_i$ on each arc $h_i$. As  every point in $\partial A\cap \partial_\infty \h^2(\kappa)$ has exactly two  asymptotic arcs,  for all $\epsilon>0$, we can choose small horocycles such that $\sum_i |h_i|<\epsilon$.

	On the arcs $c_i$ belonging to $\partial A$, we know that $u=v$ and we have $Gu-Gv=\lambda \eta$ as far as $u-v > 0$ on $A$, where $\eta$ is a  vector field perpendicular to the arcs on $\partial A$.    If $\lambda = 0$ at any point, then $\nabla u = \nabla v$ and $u=v$ at this point and, by the maximum principle at the boundary, both surfaces should be the same. Consequently, $\lambda$ has sign and we can assume that $\eta = -J\alpha'$.
	
	By~\cite[Lemma~4.1]{LeR}, we have that the inequality 
	\begin{equation}\label{ineq:G}
		\bigg\langle Gu-Gv, \frac{Gu}{\sqrt{1 + |Gu|^2}}-\frac{Gv}{\sqrt{1 + |Gv|^2}}\bigg\rangle \geq 0,
	\end{equation}
	is satisfied, with equality, if and only if $\nabla u = \nabla v$. We get that along the arcs $c_i$
	\[
	\bigg\langle Gu-Gv, \frac{Gu}{\sqrt{1 + |Gu|^2}}-\frac{Gv}{\sqrt{1 + |Gv|^2}}\bigg\rangle = \lambda \bigg\langle \eta,  \frac{Gu}{\sqrt{1 + |Gu|^2}}-\frac{Gv}{\sqrt{1 + |Gv|^2}}\bigg\rangle
	\]
	\noindent and we deduce that $\mathcal{F}(\Sigma_u, c_i)-\mathcal{F}(\Sigma_v, c_i)$ has sign for every $i=1,\dots,n$. Therefore,
	\begin{align*}
		0=\mathcal{F}(\Sigma_u, \alpha)-\mathcal{F}(\Sigma_v, \alpha) & = \sum_i (\mathcal{F}(\Sigma_u, c_i)-\mathcal F(\Sigma_v,c_i)) + \sum_i (\mathcal{F}(\Sigma_u, h_i)-\mathcal{F}(\Sigma_v, h_i))
		\\
		&<\sum_i (\mathcal{F}(\Sigma_u, c_i)-\mathcal{F}(\Sigma_v, c_i))+\epsilon.
	\end{align*}

	\noindent Choosing $\epsilon>0$ small enough, we obtain that  $\mathcal{F}(\Sigma_u, c_i)-\mathcal{F}(\Sigma_v, c_i)$ must vanish for all $i=1,\dots,n$ and we have that $\langle \eta, \frac{Gu}{\sqrt{1+|Gu|^2}}-\frac{Gv}{\sqrt{1+|Gv|^2}}\rangle=0 $ along $c_i$. We have that the equality in \eqref{ineq:G} holds, and we get a contradiction with the maximum principle in the boundary since $u=v$ and $\nabla u = \nabla v$ along the curves $c_i$.
\end{proof}
\begin{figure}[htb]
	\begin{center}
		\includegraphics[height=5cm]{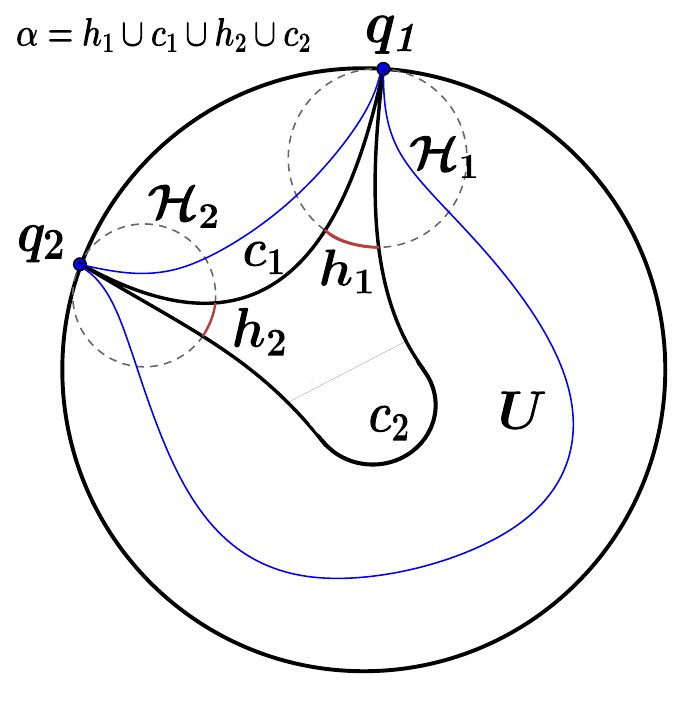}
	\end{center}
	\caption{The domain $U$ (blue) and the domain $A$ (black)  in the proof of Proposition~\ref{prop:G-M}.}
	\label{Fig-Collin_Krust}
\end{figure}
\begin{remark}
Observe that The Generalized Maximum Principle~\ref{prop:G-M} also extends to possible infinite boundary values of the graphs along the geodesic arcs of $\partial U$, see~\cite[Remark 6]{CR}.
\end{remark}

\subsection{The umbrella, the surface $\mathcal I$ and the helicoids $\mathcal H_{\infty,a_2}$ and $\mathcal H_{a_1,\infty}$.}\label{subsec:relevant surfaces}
\begin{itemize}
	\item The umbrella $\mathcal U_p$ is the minimal surface composed of all horizontal geodesics starting at $p$. The umbrella's angle function $\nu$ only takes the value $1$ at $p$. For $\kappa\leq 0$, the umbrella centered at the origin $(0,0,0)$  is the graph of the function $z=0$ in the cylinder. For $\tau>0$, the graph of the umbrella centered in $(0,y_0,0)$ with $y_0>0$, is positive in $\{x<0\}$ and negative in $\{x>0\}$, see Figure~\ref{Fig-UM-I}.
	
	\item The surface $\mathcal I$ is the minimal  surface composed of all horizontal geodesics perpendicular to a horizontal geodesic, called the axis of $\mathcal I$. The angle function $\nu$ of $\mathcal I$ is only equal to $1$ along the axis. For $\kappa\leq 0$, the surface $\mathcal I$ with axis in $\{y=0,z=0\}$ in the cylinder model is the graph of the function (see~\cite[Section 2.2]{CMR})
	 \[z(x,y)= \left\{ \begin{array}{lcc}
		\tau x y &  \mathrm{if}\; \kappa = 0, \\
		\\ \frac{2\tau}{\kappa} \arctan\frac{2xy}{\frac{4}{\kappa}+x^2-y^2} &  \mathrm{if}\; \kappa < 0,
	\end{array}
	\right.\]
	\noindent 
 see also~\cite[Section 3 (3)]{Cas} for description in the half space model with $\kappa=-1$. 

\begin{figure}[htb]
	\begin{center}
		\includegraphics[height=8cm]{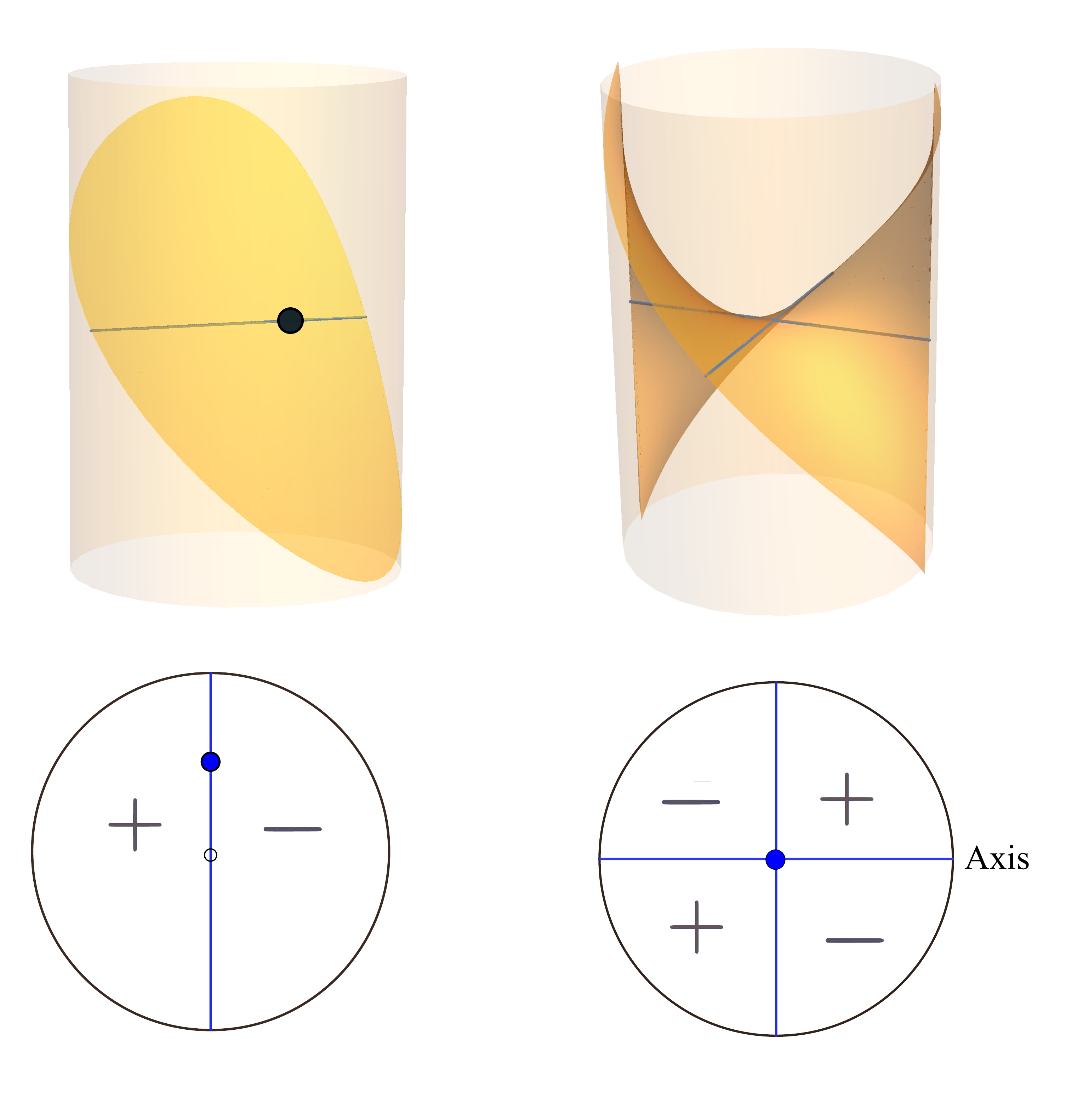}
	\end{center}
	\caption{The umbrella centered at $(0,y_0,0)$ and its projection (left) and the surface $\mathcal I$ and its projection (right) for $\kappa<0$ and $\tau>0$. }
	\label{Fig-UM-I}
\end{figure}
We can express $\mathcal I$ in euclidean  polar coordinates $(r,\theta)$, that is, $x=r \cos(\theta)$ and $y=r\sin(\theta)$, where $r$ is the euclidean distance to the origin in $\M^2(\kappa)$ (identified with a disk of radius $\frac{2}{\sqrt{-\kappa}}$) and $\theta$ is the angle formed with the axis $\{y=0\}$. 
\begin{equation}\label{eq:superficie I}
	z(r,\theta)= \left\{ \begin{array}{lcc}
	\frac{\tau}{2} r^2 \sin(2\theta) &  \mathrm{if}\; \kappa = 0, \\
	\\ \frac{2\tau}{\kappa} \arctan\frac{\kappa r^2\sin(2\theta)}{ (4 +r^2 \kappa \cos(2\theta))} &  \mathrm{if}\; \kappa < 0,
\end{array}
\right.
\end{equation}
\noindent  Notice that, when $\kappa < 0$, 
  \[\lim_{r\to \frac{2}{\sqrt{-\kappa}}}z(r,\theta) = -\frac{2\tau}{\kappa}(\frac{\pi}{2}-\theta),\ \text{for $\theta\in (0,\pi)$ } \]    and
   \[\lim_{r\to \frac{2}{\sqrt{-\kappa}}}z(r,\theta) = -\frac{2\tau}{\kappa}(-\frac{\pi}{2}+\theta),\ \text{for $\theta\in (-\pi,0)$ }.  \]   When $\kappa = 0$, if $\sin(2\theta)\neq 0$ we have $\lim_{r\to \infty }z(r,\theta)=\pm \infty$. Observe also that for $\theta\in(0,\frac \pi 2)$ the surface $\mathcal I$ is radially increasing.

\item The horizontal helicoids $\mathcal H_{\infty,a_2}$ and $\mathcal H_{a_1,\infty}$ with $a_1,a_2>0$ are two family of complete properly embedded minimal surfaces in $\Nil$ foliated by straight lines orthogonal to a horizontal geodesic $\Gamma$, see~\cite{CMR,DH}.  Let $S_2$ (resp. $S_1$) be the strip of width  $a_2$ (resp. $a_1$)  with edges $l_i$, $i=1,2,3$, being $l_2$ the edge of length $a_2$ (resp. $l_1$ the edge of length $a_1$) and $l_1$ (resp. $l_2$) and $l_3$ parallel edges orthogonal to $l_2$ (resp. $l_1$) such that the boundary  $\partial S_2 =l_1\cup l_2\cup l_3$ (resp. $\partial S_1 = l_3 \cup l_1 \cup l_2$) is traveled in a negative sense, where the union $\cup$ is written down in the same order that we travel the boundary. The fundamental piece of $\mathcal H_{\infty,a_2}$ (resp. $\mathcal H_{a_1,\infty}$) can be seen as a solution of a Jenkins-Serrin problem over $S_2$ (resp. $S_1$) with boundary values $0$ over $l_1$ and $l_2$ and $+\infty$  over $l_3$. These helicoids correspond with the family $\mathcal H_\mu$ for $|\mu|>\frac{1}{2}$ of \cite[Section 3.2]{CMR}.

\end{itemize}
\subsection{Preliminaries about conjugation}\label{subsec:conj}

We mention here a brief introduction about the conjugate technique that we will use to construct the $H$-surfaces in $\h^2\times\R$ of Section~\ref{sec:construction}. We refer to~\cite{CMT} and the references therein for more detailed description of the technique.

Daniel~\cite{Dan}, and Hauswirth, Sa Earp and Toubiana~\cite{HST} discovered a Lawson-type isometric correspondence between a simply connected minimal immersion $\widetilde\phi: \Sigma\to \E(4H^2-1,H)$ and a $H$-immersion $\phi:\Sigma\to \h^2\times\R$. The fundamental data $(A,T,\nu)$ of the $H$-immersion $\phi:\Sigma\to \h^2\times\R$, being  $A$ the shape operator, $\nu=\langle N,\xi\rangle$ the angle function and $T$ the tangent part of the Killing vector field $\xi$,   are related with the fundamental data $(\widetilde A,\widetilde T,\widetilde \nu)$ of the minimal-immersion $\widetilde\phi:\Sigma\to E(4H^2-1,H)$ by
\begin{equation}
	(A,T,\nu)=(J\widetilde A +H\cdot \text{id},J\widetilde T,\widetilde \nu).
\end{equation} 
Throughout the text we will write $\Sigma$ and $\widetilde\Sigma$ to refer to the conjugate (possibly non-embedded) surfaces. Our initial minimal piece $\widetilde\Sigma\subset\E(4H^2-1,H)$ is going to be a solution to a Jenkins-Serrin  problem, that is, a Dirichlet problem with possibly asymptotic values $\pm\infty$ over geodesics in $\M^2(4H^2-1)$. Moreover, the boundary of $\widetilde\Sigma$ will be composed of horizontal and vertical geodesics of $\E(4H^2-1,H)$ and the asymptotic boundary will be composed of vertical ideal geodesics (only in the case of $0\leq H<\frac{1}{2}$) and horizontal ideal geodesics in $\mathbb M^2(4H^2-1)\times\{\pm \infty\}$. In the following lemmas, we describe the conjugate curves of horizontal and vertical geodesics. 

\begin{lemma}\cite[Lemma~3.6]{CMT}\label{lemma:horizontal}
	If $\widetilde \gamma\subset\partial \widetilde \Sigma$ is a horizontal geodesic, then the conjugate curve $\gamma\subset\partial\Sigma$ is  contained in a vertical plane of $\h^2\times\R$ that the immersion meets orthogonally. Moreover, if $\gamma=(\beta,z)\subset\h^2\times\R$, then $|\beta'|=|\nu|$ and $|z'|=\sqrt{1-\nu^2}$.
\end{lemma}   

Assume now that $\widetilde\gamma:I\to\partial \widetilde \Sigma$ is a vertical geodesic parameterized such that $\widetilde\gamma'=\xi$ and write the normal $\widetilde N_{\widetilde\gamma}=\cos(\theta)E_1+\sin(\theta) E_2$ for some function $\theta\in C^\infty(I)$ called \emph{the angle of rotation of $\widetilde N$ along $\widetilde \gamma$}.

\begin{lemma}\cite[Lemma~3.7]{CMT}\label{lemm:vertical}
	If $\widetilde \gamma\subset\partial \widetilde \Sigma$ is a vertical geodesic, then the conjugate curve $\gamma\subset\partial\Sigma$ is contained in a horizontal plane  $\h^2\times\{z_0\}$ that the immersion meets orthogonally. 
	\begin{enumerate}
		\item The curve $\gamma$ has geodesic curvature $k_g=2H-\theta'$ with respect to $N$.
		\item Assume that $\nu>0$ in the interior of $\Sigma$ and let $\widetilde \Omega$ and $\Omega$ be the (possibly non-embedded) domains over which $\widetilde\Sigma$ and $\Sigma$ project as multigraphs, then:
		\begin{itemize}
			\item If $\theta' > 0$, then $J\widetilde\gamma'$ (resp. $J\gamma$) is a unit outer conormal to $\widetilde\Sigma$ along $\widetilde\gamma$ (resp. $\gamma$), $N$ points to the interior of $\Omega$ along $\gamma$ and $\Sigma$ lies in $\h^2\times (-\infty, z_0]$ locally around $\gamma$.
			\item If $\theta'<0$, then $J\widetilde\gamma'$ (resp. $J\gamma'$) is a unit inner conormal to $\widetilde\Sigma$ along $\widetilde\gamma$ (resp. $\gamma$), $N$ points to the exterior of $\Omega$ along $\gamma$ and $\Sigma$ lies in $\h^2\times [z_0,+\infty)$ locally around $\gamma$.
		\end{itemize}
	\end{enumerate}
\end{lemma}   

Let us consider the half-space model for $\h^2\times\R$, whose metric is given by $ds^2=y^{-2}(dx^2+dy^2)+dz^2$ and also consider the positively oriented orthonormal frame $\{E_1=y\partial_x,E_2=y\partial_y,E_3=\partial_z\}$. Let $\gamma:I\to \h^2\times\{z_0\}$ be the conjugate curve  of a vertical geodesic $\widetilde\gamma$ in $\E(4H^2-1,H)$ parameterized as $\widetilde \gamma'=\xi$. Since $\gamma$ is contained in a horizontal plane  there exists a smooth function $\psi\in C^\infty(I)$ such that $\gamma'(t) = \cos(\psi(t))E_1 + \sin(\psi(t))E_2$. The function $\psi$ is called \emph{the angle of rotation of $\gamma$ with respect to a  foliation by horocycles}, it is related with the function $\theta$ by the next Equation (see~\cite[Remark 3.8]{CMT}):
\begin{equation}\label{eq:psi}
	\psi'+\cos(\psi) = \theta' - 2H.
\end{equation}
\begin{remark}\label{Remark:signo}
	In Formula~\eqref{eq:psi} we are assuming that the curve $\gamma$ is parameterized in the direction such that $\widetilde \gamma'=\xi$ and the angle of rotation with respect to the horocycle foliation is measured with regard to the orientation given by $\gamma'$  (the unit tangent of the conjugate curve). However, if we measure the angle of rotation with respect to the horocycle foliation using the contrary orientation for $\gamma$, formula~\eqref{eq:psi} changes to  $-\psi'-\cos(\psi) = \theta' - 2H$.
\end{remark}
The following result describes the conjugate curves of the asymptotic boundary of a graph solution to a Jenkins-Serrin problem.

\begin{lemma}\cite[Corolary~2.4]{CMR}\label{lem:geodesic-asym}
	Assume that $4H^2-1<0$ and let $\widetilde\Sigma\subset \E(4H^2-1,H)$ be a solution of a Jenkins-Serrin problem with asymptotic boundary consisting of vertical and horizontal ideal geodesics. Let $\Sigma\subset\h^2\times\R$ be the conjugate $H$-multigraph.
	\begin{itemize}
		\item Ideal vertical geodesics in $\partial_\infty\widetilde\Sigma$ become  ideal horizontal curves in $\partial_\infty\Sigma$ with constant curvature $\pm 2H$ in $\h^2\times\{\pm\infty\}$.
		\item Ideal horizontal geodesics in $\partial_\infty\widetilde\Sigma$ become ideal vertical geodesics of $\partial_\infty\Sigma$.
	\end{itemize}
\end{lemma}

\section{Conjugate construction of $(H,k)$-noids and $(H,k)$-nodoids with genus one.}\label{sec:construction}
\subsection{The initial minimal graph of the conjugate construction}
We will describe a family of minimal graphs in $\mathbb E(4H^2-1,H)$ for $0<H\leq \tfrac 12$ and use this family in our conjugate construction to obtain the desired  $(H,k)$-noids and $(H,k)$-nodoids with genus one inspired by the results of~\cite{CM}.

We consider the geodesic triangle $\widetilde\Delta(a_1,a_2,\varphi)\subset \mathbb{M}^2(4H^2 -1)$ with vertexes $p_0$, $p_1$ and $p_2$. We  assume that $p_0 =(0,0)$ and the sides $l_1=\overline{p_0p_1}$, $l_2=\overline {p_0p_2}$ have lengths $a_1\in(0,\infty]$ and $a_2 \in(0,\infty]$ respectively (not both equal to $\infty$). The angle $0<\varphi<\tfrac \pi 2$ is the counter-clockwise oriented angle in $p_0$ going from the side $l_1$ to  $l_2$, see Figure~\ref{Fig-Pieza-Inicial} (down).  We call $l_3=\overline{p_1p_2}$.

 In the case $0<H<\tfrac 12$, if $a_1=\infty$ (resp. $a_2=\infty$), we have that $\widetilde\Delta(a_1,a_2,\varphi)\subset \h^2(4H^2-1)$ is a semi-ideal triangle with an ideal vertex $p_1$ (resp. $p_2$).  In the case $H = \frac 12$, the domain $\widetilde\Delta(a_1,a_2,\varphi)$ is contained in $ \mathbb{R}^2$ and we do not have an asymptotic boundary. Consequently,  if $a_1=\infty$ (resp. $a_2=\infty$)   the point $p_1$ (resp. $p_2$) disappears and $\widetilde\Delta(a_1,a_2,\varphi)$ is  a strip defined by the geodesic $l_2$ (resp. $l_1$) and the parallel rays $l_1$ (resp. $l_2$) and $l_3$. The latter rays are defined as the limits of the sides $l_1$ (resp. $l_2$) and $l_3$ where  $a_1\to\infty$ (resp. $a_2\to\infty$), see  Figures~\ref{FIG-Conj-Knoids} and~\ref{FIG-Conj-Knodoids} for the cases $a_1=\infty$ and $a_2=\infty$ respectively.

\begin{lemma}\label{lem:existence}
	There exists a unique minimal graph $\widetilde \Sigma_\varphi(a_1,a_2,b)\subset\E(4H^2-1,H)$   solution to the Jenkins-Serrin problem in $\E(4H^2-1,H)$ over $\widetilde\Delta(a_1,a_2,\varphi)$ with boundary values $0$ over $l_1$, $b\in \R$ over $l_2$ and $+\infty$  over $l_3$.
\end{lemma}
\begin{proof}
	 If $a_1$ and $a_2$ are both finite then the results follow from~\cite{Younes} for $H<\tfrac 1 2 $ and from~\cite{Pinheiro} for $H=\frac{1}{2}$ (see also~\cite{D-PMN} for more general results in a Killing submersion).
	 
	  We deal with the case $a_1 = \infty$ (resp. $a_2=\infty$). We consider a sequence $\{q_n\}_{n\in \N}\subset l_1$ (resp. $\{q_n\}_{n\in \N}\subset l_2$)  with $n$ being the distance of $q_n$ to $p_0$. By~\cite{Pinheiro,Younes,D-PMN}, for any $q_n\in l_1$ (resp. $q_n\in l_2$),  there exists a minimal graph $\widetilde\Sigma_n:=\widetilde\Sigma_\varphi(n,a_2,b)$ (resp. $\widetilde\Sigma_n:=\widetilde\Sigma_\varphi(a_1,n,b)$)  over the triangle with vertexes $q_n$, $p_0$ and $p_2$ (resp. $p_1$) with  boundary data $0$ over $\overline{p_0q_n}$ (resp. $\overline{p_0p_1}$), $b$ over $\overline{p_0p_2}$ (resp. $\overline{p_0q_n}$) and $+\infty$  over $\overline{p_2q_n}$ (resp. $\overline{q_n p_1}$). Comparing their boundary values, Proposition~\ref{prop:G-M} ensures, that the graphs defining $\widetilde\Sigma_n$ form a decreasing  sequence of functions.

	To prove the existence of these surfaces, we find upper and lower bounds (in each case) that ensure that the limit surface takes the desired limit values. We consider $\widetilde\Sigma'$ the minimal graph with boundary values $\min\{0,b\}$ over $l_1\cup l_2$ and $+\infty$ over $l_3$. The existence of the surface $\widetilde\Sigma'$ is proved in \cite[Lemma 3.2]{CMR} for $H<\tfrac1 2$ and in~\cite[Lemma 3.6]{CMR} for $H=\frac{1}{2}$ when $\varphi=\frac\pi k$. For $0<\varphi<\frac{\pi}{2}$ the same argument works using the same barriers in the limit process.  The surface $\widetilde\Sigma'$ takes the values $+\infty$ over $l_3$ and it is below  every $\widetilde\Sigma_n$ by Proposition~\ref{prop:G-M}. Therefore $\widetilde\Sigma_n$ converges  to a minimal graph $\widetilde\Sigma_\infty$ that takes the desired boundary values since we have a lower  bound $\widetilde\Sigma'$ taking the asymptotic value $+\infty$ over $l_3$ and an upper bound given by any graph $\widetilde\Sigma_n$.

	Finally, we notice that if $0 < H < \tfrac 12$, the uniqueness is a consequence of Proposition~\ref{prop:G-M}. Uniqueness for $H=\tfrac 1 2$ is understood as the uniqueness of limit of the sequence $\widetilde\Sigma_n$.
\end{proof}

\begin{remark}
	For $H=\frac 1 2$ the Generalized Maximum Principle~\ref{prop:G-M} does not apply, however the surfaces $\widetilde \Sigma_\varphi(\infty,a_2,b)$ and $\widetilde \Sigma_\varphi(\infty,a_2,b)$ are obtained as the limit of monotonous graph over compact domains, where the maximum principle apply. Then if we have two monotonous sequences of graphs ordered then the limits is also ordered. Using this we will be able to compare different  limits graphs  regarding only their values at the boundary.
\end{remark}

\begin{remark}
Observe that  the surfaces $\widetilde \Sigma_\varphi(a_1,a_2,b)$ and $\widetilde \Sigma_\varphi(a_2,a_1,b)$ are only congruent for $H=0$. In general we have that $\widetilde\Sigma_\varphi(a_1,a_2,b)$ is congruent to the minimal graph over $\widetilde\Delta(a_2,a_1,\varphi)$ that takes the asymptotic boundary values $0$ over $l_1$, $b$ over $l_2$ and $-\infty$ over $l_3$. We will denote this surface as $\widetilde\Sigma_\varphi^-(a_1,a_2,b)$.
\end{remark}

The  boundary of the minimal graph $\widetilde \Sigma_\varphi(a_1,a_2,b)$ consists of:
\begin{itemize}
	\item The horizontal geodesics $\widetilde h_1$ and $\widetilde h_2$  projecting onto the sides $l_1$ and $l_2$ and the asymptotic horizontal geodesic $\widetilde h_3\subset \mathbb M^2(4H^2-1)\times\{+\infty\}$ projecting onto $l_3$, see Figure~\ref{Fig-Pieza-Inicial}.
	\item The vertical geodesics $\widetilde v_i$ contained in $\pi^{-1}(p_i)$ for $i=0,1,2$, see Figure~\ref{Fig-Pieza-Inicial}. Observe that if $H<\tfrac 12$ and $a_1=\infty$ (resp. $a_2=\infty$) then $\widetilde v_1$ (resp. $\widetilde v_2$) is a semi-ideal vertical geodesic in the asymptotic boundary of $\mathbb E(4H^2-1,H)$. If $H=\tfrac{1}{2}$ and $a_1=\infty$ (resp. $a_2=\infty$), we do not have the boundary component $\widetilde v_1$ (resp. $\widetilde v_2$), see Figures~\ref{FIG-Conj-Knoids} and~\ref{FIG-Conj-Knodoids}.
\end{itemize}   

\begin{figure}[htb]
	\begin{center}
		\includegraphics[height=11cm]{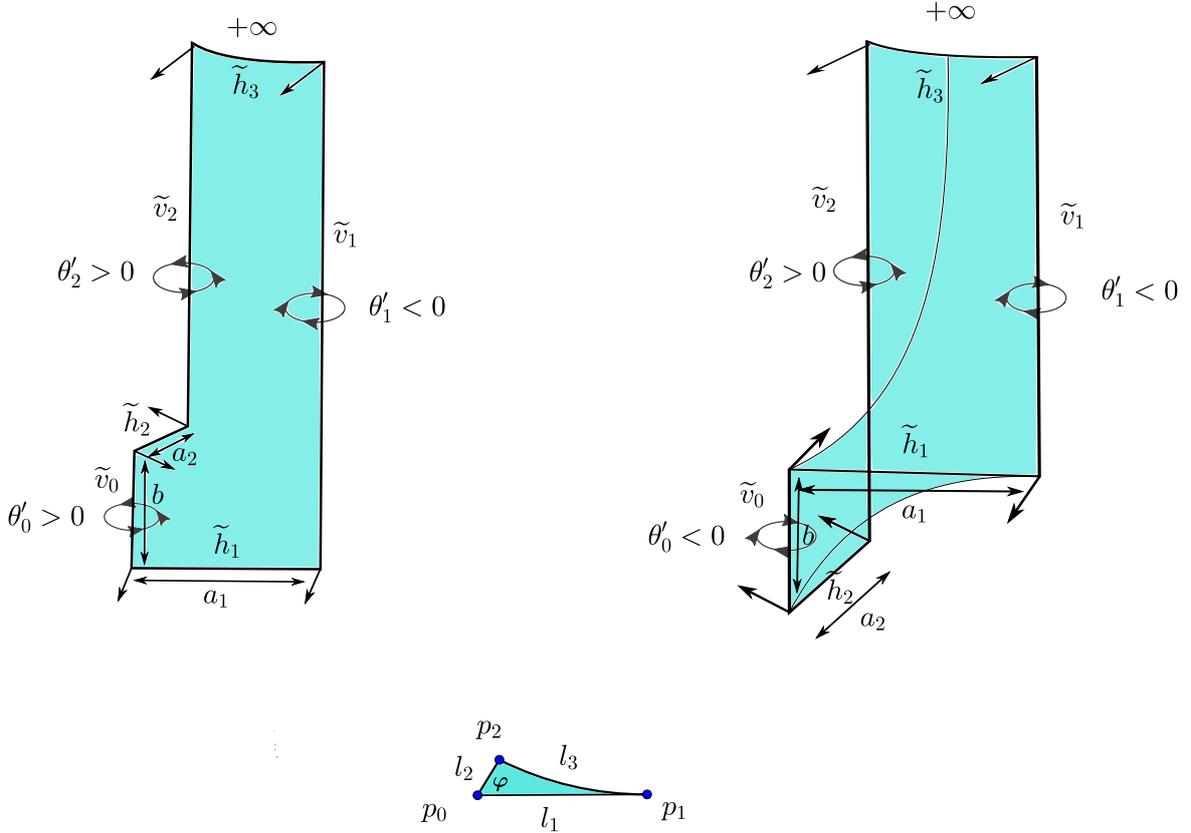}
	\end{center}
	\caption{The surfaces $\widetilde \Sigma_\varphi(a_1,a_2,b)$ for $b>0$ (top left) and for $b<0$ (top right) and their projections onto $\widetilde\Delta(a_1,a_2,\varphi)\subset\mathbb M^2(4H^2-1)$ (bottom).}
	\label{Fig-Pieza-Inicial}
\end{figure}

The following Lemma gives a description of the angle function $\nu$ of $\widetilde \Sigma_\varphi (a_1,a_2,b)$.

We define $b_\mathcal I^i(a_i,\varphi)$ for $i=1,2$ as the height of the surface $\mathcal{I}$ with axis in $\tilde h_i$ in the cylinder model, i.e.,  $b_{\mathcal I}^i(a_i,\varphi) = z(\frac {2}{\sqrt{-\kappa}}\tanh(\frac {a_i \sqrt{-\kappa}}{2}),\varphi)$  see Equation~\eqref{eq:superficie I}, and observe that $a_i$ is the hyperbolic length of $\pi(\widetilde h_i)$ $i=1,2$. The following lemma follows standard intersection comparison arguments (see~\cite[Section 3.1.2]{CMT}).

\begin{lemma}\label{lemma-nu}
 Let $\nu\geq 0$ be the angle function of $\widetilde \Sigma_\varphi(a_1,a_2, b)$. Then:
	\begin{enumerate}
		\item The  function $\nu$  only takes the value $0$  along the vertical geodesics $\widetilde v_0$,  $\widetilde v_1$ (if $a_1$ finite) and $\widetilde v_2$ (if $a_2$ finite).
		\item Let $\theta_i$ be the angles of rotation  of the normal along $\widetilde v_i$ for $i=0,1,2$. Then, $\tfrac{b}{|b|}\theta_0'>0$, $\theta'_1<0$ (if $a_1$ finite) and $\theta_2'>0$ (if $a_2$ finite).
		\item
		\begin{itemize}
			\item If $b>0$, there is exactly  one interior point $q^*$ in $\widetilde h_2$ such that $\nu(q^*)=1$.
			\item If $b\leq 0$, there are not interior points in   $\widetilde h_2$  with $\nu=1$.
		\end{itemize} 
		
		\item
		\begin{itemize}
			\item  If $b> b^1_\mathcal I(a_2,\varphi)$, there are no points with $\nu=1$ in $\widetilde h_1$. 
			\item  If $0<b\leq b^1_\mathcal I(a_2,\varphi)$, there are, at most, two interior points  $q^*_1$ and $q^*_2$ in $\widetilde h_1$ such that $\nu(q_i^*)=1$, $i=1,2$.
			\item If $b\leq 0$, there is exactly one interior point $q^*$ in $\widetilde h_1$ such that $\nu(q^*)=1$.
			
		\end{itemize}

	\end{enumerate}
\end{lemma}
\begin{proof}
	\begin{enumerate}
\item 	We have  $\widetilde\Sigma_\varphi(a_1, a_2,b)$ is a multigraph up to the boundary and $\nu$ cannot be equal to zero in the horizontal geodesics $\widetilde h_1$ and $\widetilde h_2$ by the boundary maximum  principle with respect to vertical planes. Then $\nu$ only takes the value zero along the vertical geodesics of the boundary.

\item Using that $\widetilde \Sigma_\varphi(a_1,a_2,b)$ is a graph with $\nu>0$,	it follows easily by looking at the normal in the intersection between the horizontal and vertical geodesics of the boundary (see Figure~\ref{Fig-Pieza-Inicial})  and taking into account that the normal along the vertical segments rotates monotonically. 

\item
Assume that $b>0$. As $\widetilde \Sigma_\varphi(a_1,a_2,b)$  is a vertical graph and we are assuming that the angle function $\nu$ is positive in the interior, we deduce that the normal in $\widetilde h_2\cap \widetilde v_0$  is horizontal and points to the interior of $\widetilde \Sigma_\varphi(a_1,a_2,b)$, i.e., if we divide the space by the vertical plane that contains $\widetilde h_2$, the horizontal normal points to the region that contains the surface, see Figure~\ref{Fig-Pieza-Inicial}. In the same way, we deduce that the horizontal normal at $\widetilde h_2\cap \widetilde v_2$ points to the exterior of $\widetilde \Sigma_\varphi(a_1,a_2,b)$, i.e., if we divide the space by the vertical plane that contains $\widetilde h_2$, the horizontal normal points to the region that does not contain the surface. Then, by continuity, there exists a point $q^*$ in $\widetilde h_2$ where $\nu(q^*)=1$.

We will show that there is at most one point of $\widetilde h_2$ where $\nu=1$. Assume by contradiction that there exist $q_1$, $q_2\in\widetilde h_2$  where $\nu(q_1)=\nu(q_2)=1$. We consider the surface $\mathcal I$ with axis containing $\widetilde h_2$. If we reflect $\widetilde\Sigma_\varphi(a_1,a_2,b)$ over $\widetilde h_2$, the surface $\mathcal I$ is tangent to the reflected surface in the interior points  $q_1$ and $q_2$,  therefore the intersection of both surfaces is a equiangular system of at least two curves, see for instance~\cite{Lipman}. Then, for each $q_i$ there is a curve $c_i$   through $q_i$ in $\widetilde\Sigma_\varphi(a_1,a_2,b)\cap \mathcal I$ different from $\widetilde h_2$. 
Since $c_i$ can only end in $\widetilde h_2$, and at only one point in $(\widetilde h_1\cup\widetilde v_1)\cap \mathcal I$ (the surface $\mathcal I$ is radially decreasing in this region), then the curves $c_i$ joint with $\widetilde h_1$ necessarily enclose a compact loop or a domain in the hypothesis of Proposition~\ref{prop:G-M}. Applying Proposition~\ref{prop:G-M} to $\widetilde\Sigma_\varphi(a_1,a_2,b)$
 and $\mathcal I$ in that region, we achieve a contradiction.

 Suppose that $b\leq 0$ and there is a point $q^*\in \widetilde h_2$ such that $\nu(q^*)=1$. In that case, we have that the surface $\mathcal I$ with axis in $\widetilde h_2$ is below $\widetilde\Sigma_\varphi(a_1,a_2,b)$ because the maximum of the height function of $\mathcal{I}$ in  $\widetilde \Delta(a_1,a_2,\varphi)$ is reached in $\widetilde h_2$. As these surfaces are tangent at $q^*$, there is a contradiction by the maximum principle at the boundary.

	\item 
Assume that $b > b^1_{\mathcal{I}}(a_2,\varphi)$ (if $a_2=\infty$, the argument is only valid for $H<\tfrac 1 2$) and there is one point $q\in \widetilde h_1$ where $\nu(q) =1$. The surface $\mathcal{I}$ with axis containing $\widetilde h_1$ is tangent to $\widetilde\Sigma_\varphi(a_1,a_2,b)$ at $q$ and $\mathcal{I}\cap \partial\widetilde\Sigma_\varphi(a_1,a_2,b) = \widetilde h_1$ since $b > b^1_{\mathcal{I}}(a_2,\varphi)$, by the maximum principle, we know there is a curve $c$ on $\widetilde\Sigma_\varphi(a_1,a_2,b)\cap \mathcal{I}$ different from $\widetilde h_1$. This curve must enclose a compact loop or a domain in the hypothesis of Proposition~\ref{prop:G-M}. 

Assume now that   $0 < b \leq b^1_\mathcal I(a_2,\varphi)$, the surface $\mathcal{I}$ with axis in $\widetilde h_1$ intersects exactly twice $\partial\widetilde \Sigma_\varphi(a_1, a_2,b)\backslash\widetilde h_1$ (once in $\widetilde h_2$ and once in $\widetilde v_2$)  for $0<b<b^1_\mathcal I(a_2,\varphi)$ and once in $\widetilde h_2\cap \widetilde v_2$ for $b=b^1_\mathcal I(a_2,\varphi)$ because $\mathcal{I}$ is radially increasing along $l_2$. Assume by contradiction that there are three points $q_1$, $q_2$ and $q_3$ in $\widetilde h_1$ such that $\nu(q_i)=1$. Then there are three curves in  $\widetilde\Sigma_\varphi(a_1, a_2,b)\cap \mathcal I$ and each curve starts at a point $q_i$. If there exists a curve that encloses a loop or a domain in the hypothesis of Proposition \ref{prop:G-M}, we have a contradiction. If such a curve does not exist, then two of these three curves end up at the same point of the boundary of $\widetilde \Sigma_\varphi(a_1,a_2,b)$ so we have a contradiction again by the maximum principle.

If $b\leq 0$, the surface $\mathcal{I}$ with axis in $\widetilde h_1$ intersects once $\partial\widetilde \Sigma_\varphi(a_1, a_2,b)$. Reasoning as before, we find a contradiction if we assume that there are two points $q_1$ and $q_2$ in $\widetilde h_1$ such that $\nu(q_i)=1$ for $i=1,2$. Either there exists a curve that encloses a loop or a domain in the hypothesis of Proposition \ref{prop:G-M}, so we have a contradiction. On the other hand, the normal points to opposite directions in the extremes of $\widetilde h_1$, see Figure~\ref{Fig-Pieza-Inicial}. By continuity, there exists exactly one point $q^*$ in  $\widetilde h_1$ such that $\nu(q^*)=1$.

 All the results hold true for $H=\tfrac 1 2$ using the maximum principle for bounded domains for $a_1, a_2<\infty$.   For $a_1=\infty$ or $a_2=\infty$ we use uniqueness result of Del Prete for the surface $\mathcal I$ in $\Nil$ see~\cite[Lemma 5.1]{Prete} instead of the Generalized Maximum Principle.  
\qedhere
\end{enumerate}
\end{proof}

Let  $\mathcal F_{\widetilde h_i}(b):=\mathcal F(\widetilde \Sigma_\varphi(a_1,a_2,b),\widetilde h_i)=\int_{\widetilde h_i}\langle -J\widetilde h_i',\xi\rangle$  be the flux with $\widetilde h_i$ parameterized so that $-J\widetilde h'_i$ is the inward conormal vector for $i=1,2$, see Proposition~\ref{prop:Flux}.
In the sequel, we are going to assume that $a_1=\infty$ (resp. $a_2=\infty$), and $(a_2,\varphi)\in \Omega_2$ (resp. $(a_1,\varphi)\in \Omega_1$) where
\begin{equation}\label{eq:Omega}
	\Omega_i = \{(a_i,\varphi)\in \R^2: 0 < \varphi < \tfrac{\pi}{2}, 0 < a_i < a_{\mathrm{max}}(\varphi)\},
\end{equation}
\noindent and $a_\mathrm{max}(\varphi)=2 \arctanh(\cos(\varphi))$ for $0<H<\tfrac 1 2$ and $a_\mathrm{max}(\varphi)=+\infty$ for $H=\tfrac 1 2$. This condition means that for $0<H<\tfrac 1 2$ the angle in $p_2$ (resp. $p_1$) is greater than $\varphi$ while $0<a_1<a_\mathrm{max}(\varphi)$ (resp. $0<a_2<a_\mathrm{max}(\varphi)$). This is a necessary condition in the proof of Lemma~\ref{lemma:P1}. For $0<H<\tfrac 1 2$, we will define also $a_{\mathrm{emb}}(\varphi)=\arcsinh(\cot(\varphi))$, that is, the value of the parameter $a_1$ (resp. $a_2$) for which the angle in $p_2$ (resp. $p_1$) is exactly $\frac{\pi}{2}$. This value is related with the embeddedness of the conjugate surface, see~Remark~\ref{REMARK-Emb}.
\begin{lemma}\label{Lemma:Flujo} The following statements hold true:
	\begin{itemize}
		\item  If $a_1=\infty$  and $a_2\in\Omega_2$ the function $\mathcal F_{\widetilde h_2}$  is strictly decreasing. Moreover, for all $0<H\leq\tfrac 12$, we have $\mathcal F_{\widetilde h_2}(0)>0$ and $\mathcal F_{\widetilde h_2}(b)<0$ for $b>0$ large enough. 
		\item   If $a_2=\infty$  and $a_1\in\Omega_1$  the function $\mathcal F_{\widetilde h_1}$  is strictly increasing. Moreover, for $0<H<\tfrac 1 2$ we have $\mathcal F_{\widetilde h_1}(b_\mathcal{I}^1(\infty,\varphi))>0$     and $\mathcal F_{\widetilde h_1}(-b)<0$ for $b>0$ large enough and for $H=\tfrac 12$ we have $\mathcal F_{\widetilde h_1}(b)>0$     and $\mathcal F_{\widetilde h_1}(-b)<0$ for $b>0$ large enough. 
	\end{itemize}
\end{lemma}
\begin{proof}

Let us prove that the function $\mathcal F_{\widetilde h_2}$ (resp. $\mathcal F_{\widetilde h_1}$) is strictly decreasing (resp. increasing). Set $b_1<b_2$ and translate $\widetilde \Sigma_k^2:=\widetilde \Sigma_\varphi(\infty,a_2,b_k)$ (resp. $\widetilde \Sigma_k^1:=\widetilde \Sigma_\varphi(a_1,\infty,b_k)$), $k=1,2$, until the graphs take the value  $-b_k$ ( resp. $b_k$) over $l_1$ (resp. $l_2$) and $0$ over  $l_2$ (resp. $l_1$), then we have, by Proposition~\ref{prop:G-M}, that $\widetilde \Sigma_1^2$ (resp. $\widetilde \Sigma_1^1$ ) is above (resp. below) $\widetilde \Sigma_2^2$ (resp. $\widetilde \Sigma_2^1$ ) after this translation. We parameterize  $\widetilde h_2^k:[0,a_2]\to \E(4H^2-1,H) $ (resp. $\widetilde h_1^k:[0,a_1]\to \E(4H^2-1,H) $) by unit speed with $\widetilde h_2^k(0)\in \widetilde v_0^k$ (resp. $\widetilde h_1^k(0)\in \widetilde v_1^k$) and $\widetilde h_2^k(a_2)\in \widetilde v_2^k$ (resp. $\widetilde h_1^k(a_1)\in \widetilde v_0^k$). The maximum principle in the common  boundary allows us to compare the vertical part of the inward conormal vectors $-J(\widetilde{h}_1^k)'$ (resp. $-J(\widetilde{h}_2^k)'$) obtaining 
$\langle -J (\widetilde{h}_2^1)',\xi\rangle> \langle -J(\widetilde{h}_2^2)',\xi\rangle.$ (resp. $\langle -J (\widetilde{h}_1^1)',\xi\rangle< \langle -J(\widetilde{h}_1^2)',\xi\rangle$).
Consequently, $\mathcal F_{\widetilde h_1}$ (resp.  $\mathcal F_{\widetilde h_2}$)  is strictly decreasing (resp. increasing).

Now, we will see that $\mathcal F_{\widetilde h_2}(0)>0$ for $0<H\leq\tfrac 1 2$, and $\mathcal F_{\widetilde h_1}(b^1_{\mathcal I}(\infty,\varphi))>0$ for $0<H<\tfrac 1 2$ and $\mathcal F_{\widetilde h_1}(b)>0$ for $b>0$ large enough and $H=\tfrac 1 2$. The surface $\widetilde \Sigma_\varphi(\infty, a_2, b)$  converges to $\widetilde \Sigma_\varphi(\infty,a_2, 0)$ as $b\to 0$ by continuity. The continuity is a consequence of the unicity in Lemma~\ref{lem:existence}. We have that $\widetilde \Sigma_\varphi(\infty,a_2, 0)$ is above the surface $\mathcal I$ with axis containing $\widetilde h_2$ since the graph of $\mathcal I$ with axis in $\widetilde h_2$ takes negatives values in $\widetilde \Delta(\infty,a_2,\varphi)$ . Again, by the maximum principle at the boundary, we have that  $\langle -J\widetilde h_2,\xi\rangle> 0$ when $b=0$. On the other hand, we have that, for $H<\tfrac 1 2 $, the surface $\widetilde \Sigma_\varphi(a_1, \infty, b^1_\mathcal I(\infty,\varphi))$ is above the surface $\mathcal I$ with axis  containing $\widetilde h_1$ since $\widetilde \Sigma_\varphi(a_1, \infty, b^1_\mathcal I(\infty,\varphi))$ is above  $z=b^1_\mathcal I(\infty,\varphi)$, the maximum height of the graph of $\mathcal I$ with axis $\widetilde h_1$ in $\widetilde \Delta(a_1,\infty,\varphi)$,  then a similar argument shows that $\mathcal F_{\widetilde h_1}(b^1_{\mathcal I}(\infty,\varphi))>0$.  For $H=\tfrac 1 2 $, the surface $\Sigma_\varphi(a_1, \infty,b)$ is above the surface $\mathcal I$ in a neighborhood of $\widetilde h_1$ (and consequently $\mathcal F_{\widetilde h_1}(b)>0$) for $b>0$ large enough.

We will prove now that $\mathcal F_{\widetilde h_2}(b)$ (resp. $F_{\widetilde h_1}(-b)$) is negative for $b>0$ large enough.
Assume first that $0<H<\frac 12$ and $a_1=\infty$ (resp. $a_2=\infty)$. We start by considering the isosceles triangle $\widetilde\Delta_0^2$ (resp. $\widetilde \Delta_0^1$) with base $l_2$ (resp. $l_1$) and an ideal vertex $p_3^2$ (resp. $p_3^1$). As $a_2<a_{\mathrm{max}}(\varphi)$ (resp. $a_1<a_{\mathrm{max}}(\varphi)$),  the side $\overline{p_1p_3^2}$ (resp. $\overline{p_2p_3^1}$) intersects the side $l_1$ (resp. $l_2$). Let $\widetilde\Sigma_0^2(b)$ (resp. $\widetilde\Sigma_0^1(0)$) be the unique minimal graph over $\widetilde\Delta_0^2$ (resp. $\widetilde \Delta_0^1)$ solution to the Jenkins-Serrin problem with values $+\infty$ along $\overline{p_2 p_3}$ (resp. $+\infty$ along $\overline{p_1 p_3}$), $-\infty$ along the line $\overline{p_0 p_3^2}$ (resp. $\overline{p_0 p_3^1}$ ) and $b$ along the segment $l_2$ (resp. $0$ along the segment $l_1$), see for instance~\cite[Lemma~3.2]{CMR}. We know by~\cite[Lemma 4]{CM}, that $\widetilde{\Sigma}_0^2(b)$ (resp. $\widetilde{\Sigma}_0^1(0)$) has finite radial limit at $p_0$ along $l_1=\pi(\widetilde h_1)$ (resp. $l_2=\pi(\widetilde h_2)$) so, if $b$ is large enough,  $\widetilde{\Sigma}_0^2(b)$ (resp. $\widetilde{\Sigma}_0^1(0)$) is above $\Sigma_\varphi(\infty, a_2,b)$ (resp. $\Sigma_\varphi(a_1, \infty,-b)$) in the boundary of the common domain where both are graphs. This also happens in the interior by Proposition~\ref{prop:G-M}. We compare the conormals along the curve $\widetilde h_2$ (resp. $\widetilde h_1$) which is common to both surfaces obtaining by the boundary maximum principle that $\langle -J\widetilde h_2, \xi\rangle < \langle - J\widetilde h^0_2, \xi\rangle$ (resp. $\langle -J\widetilde h_1, \xi\rangle < \langle - J\widetilde h^0_1, \xi\rangle$). Then, as  $\widetilde\Sigma_0^2(b)$ (resp. $\widetilde\Sigma_0^1(0)$) is  axially symmetric, we get that   $\mathcal{F}_{\widetilde h_2}(b) < \mathcal{F}_{\widetilde h_2^0}=0$ (resp. $\mathcal{F}_{\widetilde h_1}(-b) < \mathcal{F}_{\widetilde h_1^0}=0$) for large $b>0$.	
Assume now that $H=\tfrac 12$ and $a_1=\infty$ (resp. $a_2=\infty$). We define the strip $\widetilde\Delta_0^2\subset \R^2$ (resp. $\widetilde\Delta_0^1\subset \R^2$ ) whose sides are formed by the edge $l_1$ (resp. $l_2$) and two parallel rays $l^0_2$ (resp. $l^0_1$) and $l^0_3$ starting at $p_2$ (resp. $p_1$) and $p_0$ respectively. These rays are perpendicular to the edge $l_2$ (resp. $l_1$) such that the interior of $\widetilde\Delta_0^2$ (resp. $\widetilde\Delta_0^1$) contains a part of the ray $l_2\subset \widetilde \Delta(\infty,a_2,\varphi)$ (resp. $l_1\subset \widetilde \Delta(a_1,\infty,\varphi)$). Let $\widetilde\Sigma_0^2(b)$ (resp. $\widetilde\Sigma_0^1(0)$) be twice the fundamental piece of the helicoid $\mathcal H_{\infty,a_2}$ (resp. $\mathcal H_{a_1,\infty}$)  which is a solution to the Jenkins-Serrin problem with boundary values $-\infty$ over $l^0_3$, $+\infty$ over $l^0_2$ (resp. $l^0_1$) and $b$ over $l_2$ (resp. $0$ over $l_1$). Observe that $\widetilde\Sigma_0^2(b)$ (resp. $\widetilde\Sigma_0^1(0)$) is axially symmetric and therefore $\mathcal F_{\widetilde h_2^0}=0$ (resp. $\mathcal F_{\widetilde h_1^0}=0$). Using similar arguments to the case $0<H<\tfrac 1 2$, we can compare the surfaces  $\widetilde\Sigma_0^2(b)$ (resp. $\widetilde\Sigma_0^1(0)$) with $\widetilde\Sigma_\varphi(\infty, a_2,b)$ (resp. $\widetilde\Sigma_\varphi(a_1, \infty,-b)$)  along the curve  $\widetilde h_2$ (resp. $\widetilde h_1$) for $b>0$ large enough, obtaining the statement.
\end{proof}

\subsection{The conjugate sister surface and the period problems}
We will now describe the conjugate sister surface of the fundamental piece $\widetilde \Sigma_\varphi(a_1,a_2,b)$ and the two period problems that arise in the construction of the $(H,k)$-noids and $(H,k)$-nodoids with genus one. 

The conjugate sister surface $\Sigma_\varphi(a_1,a_2,b)\subset \mathbb{H}^2\times\mathbb{R}$ is a multi-graph  over a (possibly non-embedded) domain $\Delta\subset \h^2$. The boundary of this surface is composed by the conjugate curves of  $\partial \widetilde\Sigma_\varphi (a_1, a_2, b)$. These curves are:

\begin{itemize}
	\item  The symmetry curves $h_1$ and $h_2$ contained in  vertical planes of symmetry and the ideal vertical half-line $h_3$ contained in $\partial_\infty \h^2\times\R$.

	\item The symmetry curves $v_0$, $v_1$ and $v_2$ are contained in horizontal planes of symmetry. Observe that if $0<H<\tfrac 1 2$ and $a_1=\infty$ (resp $a_2=\infty$) then $v_1$ (resp. $v_2$) is an ideal curve of constant curvature $2H$ in $\h^2\times\{+\infty\}$ (resp. $\h^2\times\{-\infty\}$), whose normal points to the exterior (resp. interior) of $\Delta$, see Figures~\ref{FIG-Conj-Knoids} and~\ref{FIG-Conj-Knodoids} and Lemmas~\ref{lemm:vertical} and~\ref{lem:geodesic-asym}.
\end{itemize}

	For $0<H<\tfrac 1 2 $, if $a_1=\infty$ (resp. $a_2=\infty$)  we have that the curve $\pi(h_1)$ (resp. $\pi(h_2)$) is compact since $\pi(h_1)$ (resp. $\pi(h_2)$) is a curve joining an interior point of $\h^2$ with an interior point of the curve of constant curvature  $\pi(v_1)$ (resp. $\pi(v_2))$, see Lemma~\ref{lem:geodesic-asym}. Moreover as the area of $\pi(\widetilde\Sigma_\varphi(\infty,a_2,b))$ (resp. $\pi(\widetilde\Sigma_\varphi(a_1,\infty,b))$) is finite, the sides $l_1$ (resp. $l_2$) and $l_3$ are asymptotic. Therefore, the area of $\pi(\Sigma_\varphi(\infty,a_2,b))$ (resp. $\pi(\Sigma_\varphi(a_1,\infty,b))$)  is also finite since this area is preserved by conjugation (see~\cite[Section 3.2.1]{CMR}). 
For $H=\tfrac 1 2 $, if $a_1=\infty$ (resp. $a_2=\infty$), we get that the curve $\pi(h_1)$ (resp. $\pi(h_2))$ diverges in $\h^2$, since the area of $\pi(\Sigma_\varphi(\infty,a_2,b))$ (resp. $\pi(\Sigma_\varphi(a_1,\infty,b))$ is infinite).

\begin{figure}[htb]
	\begin{center}
		\includegraphics[height=9cm]{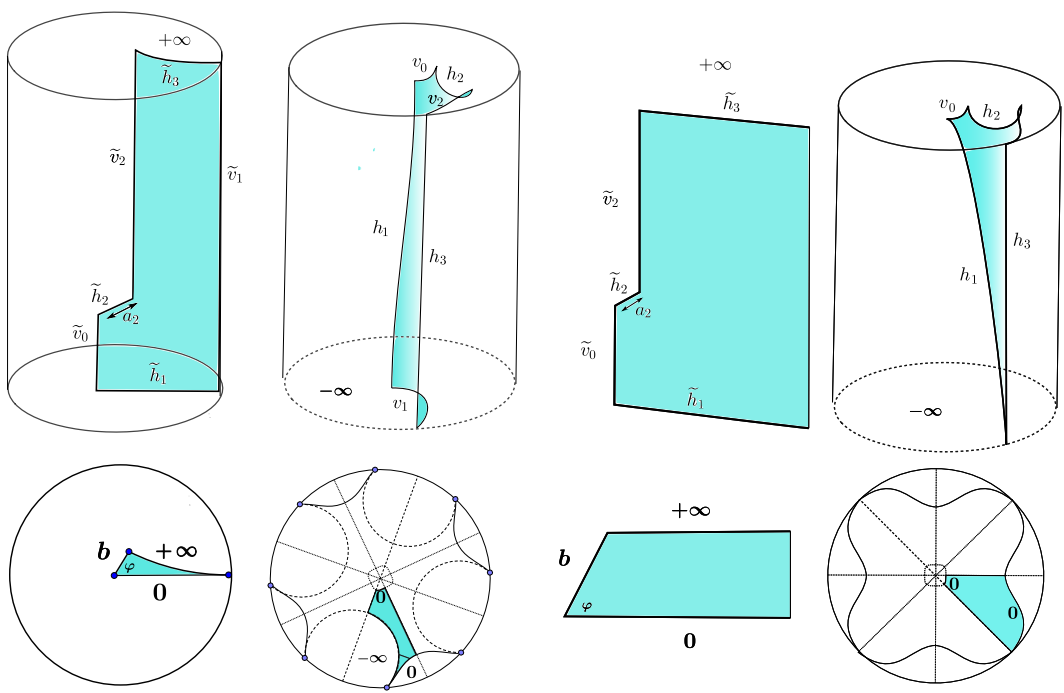}
	\end{center}
	\caption{The surface $\widetilde \Sigma_\varphi(\infty,a_2,b)$ and its conjugate sister surface $\Sigma_\varphi(\infty,a_2,b)$ with $b>0$ solving both period problems for $0<H<\tfrac 12$ (left) and $H=\frac 1 2$ (right).}
	\label{FIG-Conj-Knoids}
\end{figure}  

\begin{figure}[htb]
	\begin{center}
		\includegraphics[height=9cm]{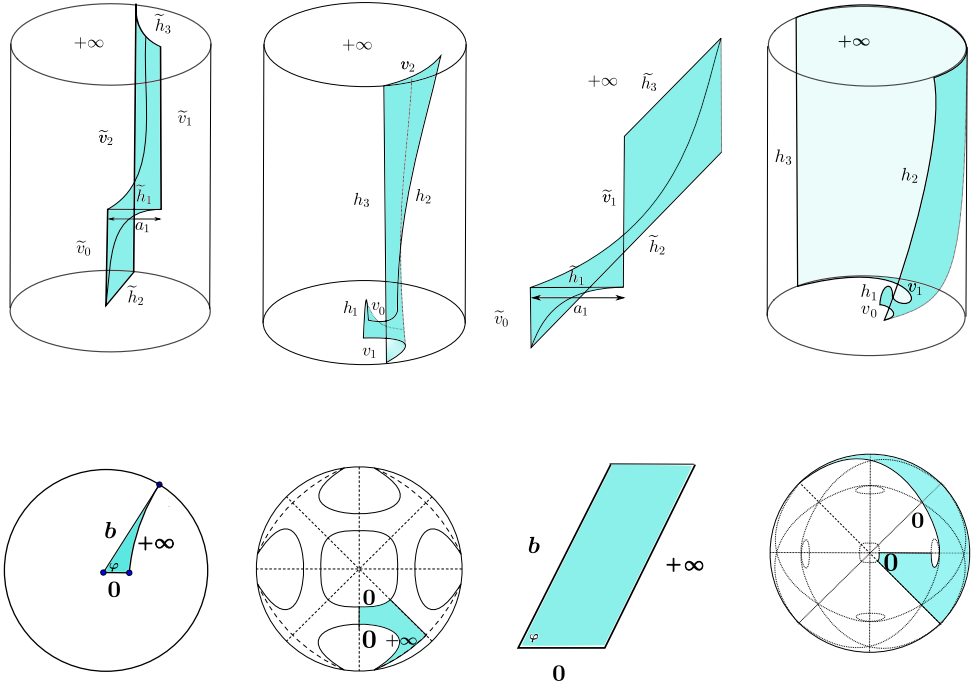}
	\end{center}
	\caption{The surface $\widetilde \Sigma_\varphi(a_1,\infty,b)$ and its conjugate sister surface $\Sigma_\varphi(a_1,\infty,b)$ with $b<0$  solving both periods problems. for $0<H<\tfrac 12$ and $H=\frac 1 2$.}
	\label{FIG-Conj-Knodoids}
\end{figure}

In the sequel we will assume that $a_1=\infty$ (resp. $a_2=\infty$). Our aim is to obtain a complete $H$-surface in $\h^2\times\R$ with genus $1$  after successive reflections over the vertical and the horizontal planes of symmetry. This is equivalent to the condition that the curves $v_0$ and $v_2$ (resp. $v_1$) lie in the same horizontal plane of $\h^2\times\R$ (first period problem) and the vertical planes of symmetry containing the curves $h_1$ and $h_2$ intersect each other with an angle $\frac{\pi}{k}$ (second period problem).
\begin{itemize}
	\item \emph{The first period function.} Assume that $a_1=\infty$ and $b>0$ (resp. $a_2=\infty$ and $b<b_\mathcal I^1(\infty,\varphi))$, as in~\cite{CM,Plehnert}, we define the first period function $\mathcal P_1^2:\Omega_2\times\R^+\to \R$ (resp. $\mathcal P_1^1:\Omega_1\times(-\infty,b_\mathcal I^1(\infty,\varphi))\to \R$) as the difference of heights between the horizontal planes containing $v_0$ and $v_2$ (resp. $v_1$), or the difference of heights of the endpoints of $h_2$ (resp. $h_1$). Parameterizing $h_2:[0,a_2]\to \h^2\times \R $ (resp. $h_1:[0,a_1]\to \h^2\times \R $) by unit speed with $h_2(0)\in v_0$ and $h_2(a_2)\in v_2$ (resp. $h_1(0)\in v_1$ and $h_1(a_1)\in v_0$), we can express the period function as:

	\begin{equation}\label{eq:fist-period}
		\mathcal P_1^i (a_i,b,\varphi)=\int_{h_i}\langle h_i',\xi\rangle=\int_{\widetilde h_i}\langle -J\widetilde h_i',\xi\rangle=F_{\widetilde h_i}(b), \;\;\; i = 1,2.
	\end{equation}

	\item \emph{The second period function.}
	\begin{itemize}
		\item Assume $a_1=\infty$ and $b>0$. 
	 We consider the half-space model for $\h^2\times\R$ and translate and rotate $\Sigma_\varphi(\infty,a_2,b)$ so that $h_1$ lies in the vertical plane $\{x=0\}$ and $v_0$ lies in the horizontal plane $\{z=0\}$. We will call $\gamma\times\R$ the vertical plane containing the symmetry curve $h_2$, where $\gamma$ is the complete extension of the geodesic $\pi(h_2)$. We identify $\h^2\times\{0\}$ with $\h^2$ and parameterize  $v_0:[0,b]\to \h^2$ by arc-length as $v_0(s)=(x(s),y(s))$ with $v_0 (0)=(0,1)$ and $v_0'(0)=-E_1$. Then we get that $x(s)<0$ for $s$ near $0$ and we denote $(x_0,y_0)=(x(b),y(b))$. Let $\psi$ be the angle of rotation with respect to the horocycle foliation (see~\eqref{eq:psi}) with initial angle $\psi(0)=\pi$ and define $\psi_0=\psi(b)$, see Figure~\ref{FIG-2-PER}. The second period function $\mathcal P_2^2:\Omega_2\times\R^+\to\R$ is defined as in \cite{CM,Plehnert}
	\begin{equation}\label{eq:second-period}
		\mathcal P_2^2(a_2,\varphi,b)=\frac{x_0\sin(\psi_0)}{y_0}-\cos(\psi_0).
	\end{equation}

\item	Assume that $a_2=\infty$ and $b<b^1_\mathcal I(\infty,\varphi)$. Aiming at defining the second period function analogously to the case $a_1=\infty$ and keeping the same orientation, we apply a reflection over the horizontal geodesic $\widetilde h_2$ to the surface $\widetilde\Sigma_\varphi(a_1,\infty,b)$ or equivalently we reflect $\Sigma_\varphi(a_1,\infty,b)$ over the vertical plane containing $h_2$. We call these surfaces $\widetilde\Sigma^-_\varphi(a_1,\infty,b)$ and $\Sigma^-_\varphi(a_1,\infty,b)$  (the conjugate $H$-surface).
Again in the half-space model of $\h^2\times\R$, we translate and rotate the surface $\Sigma_\varphi^-(a_1,\infty,b)$ so that $h_2^-$ lies in the vertical plane $\{x=0\}$ and $v_0^-$ lies in the horizontal plane $\{z=0\}$. We will call $\gamma\times\R$ the vertical plane containing the symmetry curve $h_1^-$, where $\gamma$ is the complete extension of the geodesic $\pi(h_1^-)$. We identify $\h^2\times\{0\}$ with $\h^2$ and we parameterize  $v_0^-:[0,|b|]\to \h^2$ by arc-length as $v_0^-(s)=(x(s),y(s))$ with $v_0^-(0)=(0,1)$ and $(v_0^-)'(0)=-E_1$. Then we get that $x(s)<0$ for $s$ near $0$ and we denote $(x_0,y_0)=(x(b),y(b))$. 
	This orientation coincides with the orientation that came from the election $(\widetilde v_0^-)'=\xi$ for $0<b<b^1_\mathcal I(\infty,\varphi)$ since $\theta'_0>0$ and with the contrary orientation when $b<0$ since  $\theta_0'<0$ (see~Lemma~\ref{lemma-nu}). We choose this orientation for $v_0^-$ in order to work with both cases at once and for similarity to the construction with $a_1=\infty$. Let $\psi$ be the angle of rotation of $v_0^-$ with respect to the horocycle foliation (see~\eqref{eq:psi}) with initial angle $\psi(0)=\pi$ and define $\psi_0=\psi(|b|)$, see Figure~\ref{FIG-2-PER}. The second period function $\mathcal P_2^1:\Omega_1\times(-\infty, b_\mathcal I^1(\infty,\varphi))\to\R$ is defined as in Equation~\eqref{eq:second-period}:
	\begin{equation}\label{eq:second-period-2}
		\mathcal P_2^1(a_1,\varphi,b)=\frac{x_0\sin(\psi_0)}{y_0}-\cos(\psi_0).
	\end{equation}
\end{itemize}

\begin{figure}[htb]
	\begin{center}
		\includegraphics[height=4.5cm]{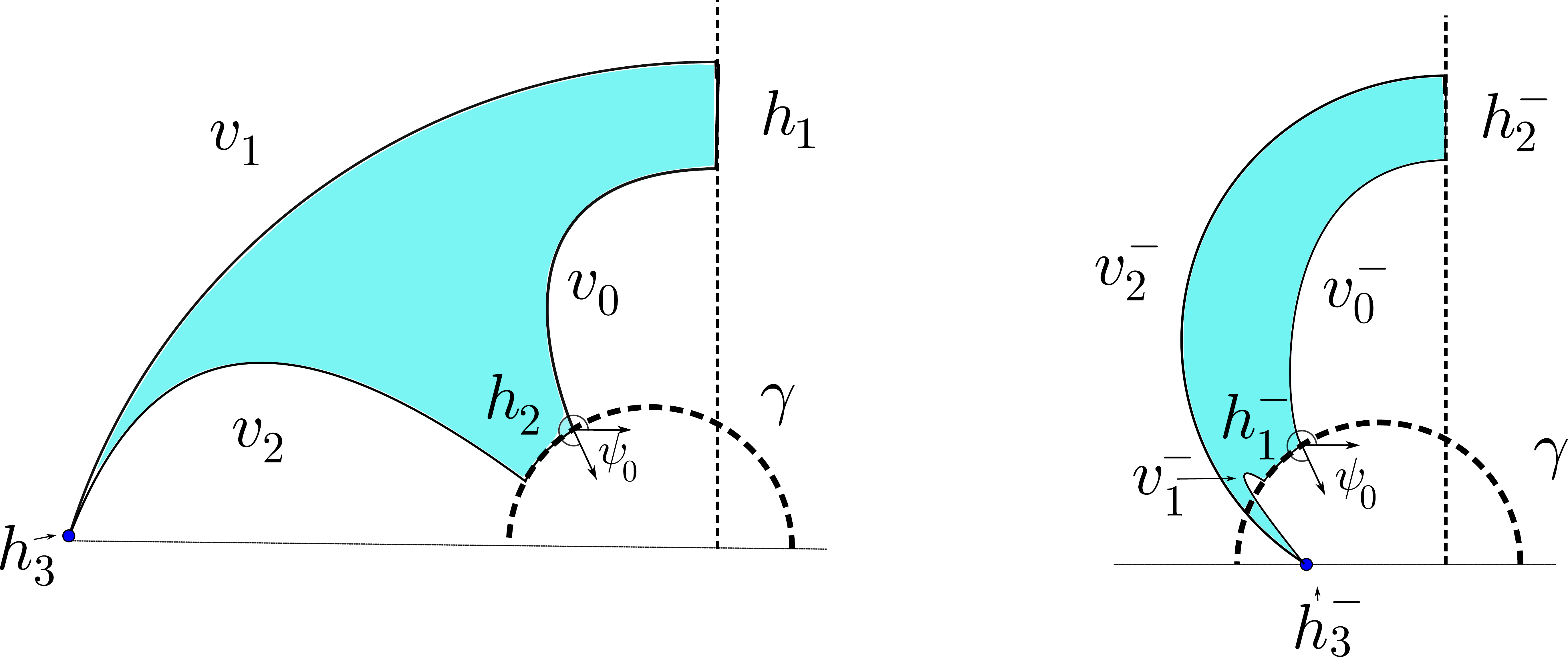}
	\end{center}
	\caption{The projection of the surfaces $\Sigma_\varphi(\infty,a_2,b)$ and $\Sigma_\varphi^-(a_1,\infty,b)$ with $0<H<\frac 1 2$ under the assumptions of the second period problem.}
	\label{FIG-2-PER}
\end{figure}  

\end{itemize}
	We will see in~Lemmas~\ref{lemma-P2} and~\ref{lemma-P2-2} that, under the assumption $0<\mathcal P_2<1$, the vertical planes containing $h_1$ and $h_2$ intersect each other with an angle $\arccos (\mathcal P_2)$.
We aim at proving that there exist parameters $(a_2,\varphi,b)$ (resp. $(a_1,\varphi,b)$) such that {$\mathcal P_1^2(a_2,\varphi,b)=0$} (resp. {$\mathcal P_1^1(a_1,\varphi,b)=0$}) and $\mathcal P_2^2(a_2,\varphi,b)=\cos(\frac{\pi}{k})$ (resp. $\mathcal P_2^1(a_1,\varphi,b)=\cos(\frac{\pi}{k})$), solving the two period problems.

\begin{lemma}\label{lemma:P1}
	If $0<H\leq\tfrac 12$ and $a_1=\infty$ (resp. $a_2=\infty$), there exists an unique function $f_2:\Omega_2\to\R^+$ (resp. $f_1:\Omega_1\to (-\infty, b_\mathcal I^1(\infty,\varphi))$) such that $\mathcal P_1^2(a_2,\varphi,f_2(a_2,\varphi))=0$ (resp. $\mathcal P_1^1(a_1,\varphi,f_1(a_1,\varphi))=0$),  for every pair $(a_2,\varphi)\in \Omega_2$ (resp. $(a_1,\varphi)\in \Omega_1$).
	Moreover:
	\begin{itemize}
		\item If $0<H<\tfrac 1 2$, then $f_2$ (resp. $f_1$) is a continuous function and 
	\begin{align*}
		&\lim_{a_2\to 0}f_2(a_2,\varphi)=0\ \ \ \ \ \ \ \ \ \  \text{(resp. }\lim_{a_1\to 0}f_1(a_1,\varphi)=0^-),\\
		&\lim_{a_2\to a_{\mathrm{max}}}f_2(a_2,\varphi)=+\infty\ \ \ \text{(resp. }\lim_{a_1\to a_{\mathrm{max}}}f_1(a_1,\varphi)=-\infty),
	\end{align*}
and given $\varphi\in(0,\frac{\pi}{2})$, $f_1(\cdot,\varphi):(0,a_{\mathrm{max}}(\varphi))\to(0,+\infty)$ (resp. $f_2(\cdot,\varphi):(0,a_{\mathrm{max}}(\varphi))\to(0,-\infty)$) is a strictly increasing (resp. decreasing) function for $a_2\geq a_{\mathrm{emb}}(\varphi)$ (resp. $a_1\geq a_{\mathrm{emb}}(\varphi)$).
	\item If $H=\tfrac 1 2$, then the function $f_2$ (resp. $f_1$) is a continuous function and satisfies 
	\begin{align*}
		&\lim_{a_2\to 0}f_2(a_2,\varphi)=0\ \ \ \ \ \ \ \ \ \text{(resp. }\lim_{a_1\to 0}f_1(a_1,\varphi)=0^-),\\
		&\lim_{a_2\to \infty}f_2(a_2,\varphi)=+\infty \ \ \ \  \text{(resp. }\lim_{a_1\to\infty}f_1(a_1,\varphi)=-\infty),\\
		&\lim_{\varphi\to \frac{\pi}{2}}f_2(a_2,\varphi)=+\infty \ \ \ \ \  \text{(resp. }\ \lim_{\varphi\to \frac{\pi}{2}}f_1(a_1,\varphi)=-\infty).
	\end{align*}
\end{itemize}
\end{lemma}
\begin{proof}
	First of all, observe that, as in~\cite[Lemma 3]{CM}, the function $\mathcal P_1^2$ (resp. $\mathcal P_1^1$) is  continuous. This function is strictly decreasing (resp. increasing) in the third parameter by Lemma~\ref{Lemma:Flujo}. Observe that, when $b=0$, the vertical part of the inward conormal is not continuous since the vertical segment disappears. However, the function $\mathcal P_1$ is continuous.
	
	Fix $(a_2,\varphi)\in \Omega_2$ (resp. $(a_1,\varphi)\in \Omega_1$). By Lemma~\ref{Lemma:Flujo} and the continuity and monotonicity of $\mathcal P_1^2$ (resp. $\mathcal P_1^1$) in the third parameter, there exists a unique $b_0^2\in(0,+\infty)$ (resp. $b_0^1\in (-\infty,b^1_{\mathcal{I}}(\infty,\varphi)))$ such that $\mathcal P_1^2(a_2,\varphi,b_0^2)=0$ (resp. $\mathcal P_1^1(a_1,\varphi,b_0^1)=0$). Therefore, we define univocally $f_2(a_2,\varphi):=b_0^2$ (resp. $f(a_1,\varphi):=b_0^1$). Moreover, the continuity of $f_2$ (resp. $f_1$) is guaranteed by its uniqueness, see also~\cite[Lemma 5]{CM}.

	 The computations of the limits of $f_i(\cdot,\varphi)$ are based on~\cite[Lemma~5]{CM}. 
	 
		We consider the case $0<H<\tfrac 1 2$. Assume by contradiction that there exists a subsequence $a_2^n\to 0$ (resp. $a_1^n\to 0$) such that $f_2(a_2^n,\varphi)$ (resp. $f_1(a_1^n,\varphi))$ converges to some $b_\infty^2\in (0,+\infty)$ (resp. $b_\infty^1\in (-\infty,0)$). Translate vertically $\widetilde\Sigma_\varphi(\infty,a_2^n,f_2(a_2^n,\varphi))$ (resp. $\widetilde\Sigma_\varphi(a_1^n,\infty,f_1(a_1^n,\varphi))$) until they take the value $-f_2(a_2^n,\varphi)$ (resp. $f_1(a_1^n,\varphi)$) over $l_1$ (resp. $l_2$) and the value $0$ over $l_2$ (resp. $l_1$). Since $a_i^n\to0$ for $i=1,2$, we can blow up the surface and the metric of $\E(4H^2-1,H)$ such that we get $a_2^n=1$ (resp $a_1^n=1$). The new sequence of rescaled surfaces converges in the $\mathcal C^k$-topology to a minimal surface $\widetilde\Sigma_\infty$ in $\R^3$. This minimal surface is a vertical graph over a strip $\widetilde\Delta (\infty,1,\varphi)\subset\R^2$ (resp. $\widetilde\Delta (1,\infty,\varphi)\subset\R^2$) bounded by the two parallel lines $l_1'$ (resp. $l_2'$) and $l_3'$ (resp. $l_3'$) and a segment  $l_2'$ (resp. $l_1'$) of length $1$ forming an angle $\varphi$ with $l_1'$ (resp. $l_2'$). Moreover, $\widetilde\Sigma_\infty$ takes the value $0$ over $l_2'$ (resp. $l_1'$), $-\infty$ over $l_1'$ (resp. $l_2'$) and $+\infty$ over $l_3'$ (resp. $l_3'$) since $b_\infty^2>0 $ (resp. $b_\infty^1<0)$. However, $\widetilde\Sigma_\infty$ cannot have first period function equal to zero since  $\widetilde\Sigma_\infty$ lies below (resp. above) a helicoid $\widetilde\Sigma_0^2$ with axis in $l_2'$ (resp. a helicoid $\widetilde\Sigma_0^1$ with axis $l_1'$), which is a graph over a half-strip and the helicoid $\widetilde\Sigma_0^2$ (resp. $\widetilde\Sigma_0^1$) has period $0$ because it is axially symmetric, see also~\cite[Figure 4]{CM}. 
		
		Let us see now that the case  $a_2 = \infty$ and $b_\infty\in (0,b_\mathcal I^1 (\infty,\varphi))$ gets into contradiction. Again we can blow up the surfaces and the metric of $\E(4H^2-1,H)$  such that we get $a_1^n=1$. The new sequence of rescaled surfaces converges in the $\mathcal C^k$-topology to a minimal surface $\widetilde\Sigma_\infty$ in $\R^3$. This minimal surface $\widetilde\Sigma_\infty$ is a vertical graph over the  strip $\widetilde \Delta(1,\infty,\varphi)\subset\R^2$ and  takes the value $-\infty$ over $l_1'$ (since $b_\infty $ is not zero), $0$ over $l_2'$ and $+\infty$ over $l_3'$. Again is easy to check that $\widetilde\Sigma_\infty$ cannot have the first period function equal to zero by applying the maximum principle.
		
		We see now that $f_1(a_1^n,\varphi)$ converges to $0$ from below when $a^n_1 \to 0$. Assume by contradiction there exists a subsequence $a_1^{\sigma(n)}$ such that $b_n:=f_1(a_1^{\sigma(n)},\varphi)>0$. Blowing up the surfaces and the metric as in the previous arguments we have that the new sequence of rescaled surfaces converges in the $\mathcal C^k$-topology to a minimal surface $\widetilde\Sigma_\infty$ in $\R^3$. This minimal surface is a vertical graph over the   strip $\widetilde\Delta(1,\infty,\varphi)\subset\R^2$  and takes the value $b_\infty=\lim_{n\to\infty}\frac{b_n}{a_n}>0$ over $l_1'$, $0$ over $l_2'$ and $+\infty$ over $l_3'$. Then, as $b_\infty>0$ the limit surface is above the plane $\{z=0\}$ and consequently the first period cannot be $0$.		
		
		Assume by contradiction that there exists a subsequence $a_2^n\to a_{\mathrm{max}}$ (resp. $a_1^n\to a_{\mathrm{max}}$)  such that $f_2(a_2^n, \varphi)\to b^2_\infty\in [0,+\infty)$ (resp.  $f_1(a_1^n, \varphi)\to b^1_\infty\in (-\infty,b_\mathcal I^1(\infty,\varphi)]$). We translate vertically the axially symmetric surface $ \widetilde\Sigma_0^2=\widetilde\Sigma^2_0(b)$ (resp. $\widetilde\Sigma_0^1= \widetilde\Sigma_0^1(0))$ (mentioned  in Lemma~\ref{Lemma:Flujo}) until it takes the value $b^2_\infty$ (resp. $b^1_\infty$) over the edge $l_2$ (resp. $l_1$). We get that the surface $\widetilde\Sigma^2_0$ (resp. $\widetilde\Sigma^1_0$) is below (resp. above) the limit surface $\widetilde\Sigma_\varphi(\infty, a_{\mathrm{max}},b^2_\infty)$ (resp. $\widetilde\Sigma_\varphi(a_{\mathrm{max}},\infty,b^1_\infty)$) and therefore  the period function $\mathcal P_1^2(a_{\mathrm{max}},\varphi,b^2_\infty)$ (resp. $\mathcal P_1^1(a_{\mathrm{max}},\varphi,b^1_\infty)$) is not zero, which is a contradiction.

		Let $0<\varphi<\frac{\pi}{2}$ and assume  by contradiction that the function $f_2(\cdot,\varphi)$ (resp. $f_1(\cdot,\varphi)$) is not strictly increasing (resp. decreasing)   for $a>a_{\mathrm{emb}}$.  Then, in both cases, there exist two numbers $\rho_1,\rho_2\in \R$ such that $a_{\mathrm{emb}}\leq \rho_1<\rho_2$ and $f_2(\rho_1,\varphi)=f_2(\rho_2,\varphi)=b_0\in(0,+\infty)$ (resp. $f_2(\rho_1,\varphi)=f_2(\rho_2,\varphi)=b_0\in(-\infty,b_\mathcal I^1(\infty,\varphi))$). Let  $\widetilde\Sigma_i=\widetilde\Sigma_\varphi(\infty,\rho_i,b)$ (resp. $\widetilde\Sigma_i=\widetilde\Sigma_\varphi(\rho_i,\infty,b)$). In this setting, the horizontal geodesic of finite length in $\widetilde\Sigma_i$ is denoted with a superindex, $\widetilde h_2^i$ (resp. $\widetilde h_1^i$), to indicate that $|\widetilde h_2^i| = \rho_i$ (resp. $|\widetilde h_2^i| = \rho_i$) for $i=1,2$. Then, both surfaces solve the first period problem and we have
	\begin{align*}
	&	\mathcal F(\widetilde\Sigma_1,\widetilde h_2^i)=\mathcal F(\widetilde\Sigma_2,\widetilde h_2^i)=0 \;\; \text{for}\ \ i = 1,2,\\
	\text{(resp. } &\mathcal F(\widetilde\Sigma_1,\widetilde h_1^i)=\mathcal F(\widetilde\Sigma_2,\widetilde h_1^i)=0 \;\; \text{for}\ \ i = 1,2.)
	\end{align*}
	
	Observe that $\widetilde h_1^1= \widetilde h_1^2$ (resp. $\widetilde h_2^1= \widetilde h_2^2$) and  $\widetilde\Sigma_1$ is above  $\widetilde\Sigma_2$  by Proposition~\ref{prop:G-M}. We choose a small horocycle $\mathcal H$ in the vertex $p_1$ (resp. $p_2$) and consider $D_\mathcal H$ its inner domain with boundary $\mathcal H$. By the maximum principle in the boundary, we can compare the vertical part of the inward conormals obtaining that
	\begin{align}\label{eq:Flux-h_3}
		& \mathcal F(\widetilde\Sigma_1,\widetilde h_1^1\backslash D_\mathcal H)>\mathcal F(\widetilde\Sigma_2,\widetilde h_1^2\backslash D_\mathcal H) \\
		\text{(resp. }&\mathcal F(\widetilde\Sigma_1,\widetilde h_2^1\backslash D_\mathcal H)>\mathcal F(\widetilde\Sigma_2,\widetilde h_2^2\backslash D_\mathcal H)\text{)}.\notag
	\end{align}
	On the other hand, traveling in the boundary with the same orientation --- the same that we chose for defining the flux over $\widetilde h_2^i$ (resp. $\widetilde h_1^i$) --- we have that 
	\begin{equation}\label{eq:sum-Flux}
		\sum_{j=1}^3\mathcal F(\widetilde\Sigma_i,\widetilde h_j^i\backslash D_\mathcal H)+\mathcal F(\widetilde\Sigma_i,\mathcal H)=0,
	\end{equation}
	where $\mathcal F(\widetilde\Sigma_i,\mathcal H)$ represents the flux that we add when we truncate with the horocycle $\mathcal H$. We have that for all $\epsilon>0$, we can  choose  a small horocycle $\mathcal H$ such that  $|\mathcal F(\widetilde\Sigma_i,\mathcal H)|<\frac\epsilon 2$. Moreover, we know by \ref{prop:Flux} $\mathcal F(\widetilde\Sigma_i,\widetilde h_3^i\backslash D_\mathcal H)=-|\pi(\widetilde h_3^i\backslash D_\mathcal H)|$  since $-J\widetilde h_3^i=-\xi$ for $i = 1,2$. Then, combining \eqref{eq:Flux-h_3} and \eqref{eq:sum-Flux}, we obtain (in both cases)
	\begin{align*}
	&|\pi(\widetilde h_3^2\backslash D_\mathcal H )|-|\pi(\widetilde h_3^1\backslash D_\mathcal H)| <|\mathcal F(\widetilde\Sigma_2,\mathcal H)- F(\widetilde\Sigma_1,\mathcal H)|<\epsilon.
	\end{align*}
	Furthermore, we have that  $|\pi(\widetilde h_3^2\backslash D_\mathcal H )|-|\pi(\widetilde h_3^1\backslash D_\mathcal H )|=c>0$ for any choice of $\mathcal H$ since $a_{\mathrm{emb}}\leq \rho_1<\rho_2$. Then choosing $\epsilon<c$, we achieve a contradiction.

	We consider now the case $H=\tfrac 1 2$. The limits of $f_i(a_i,\varphi)$ when $a_i\to 0$ and  $a_i\to\infty$ can be computed by similar arguments to those in the case $0<H<\tfrac 1 2$. Finally, we compute the limit when $\varphi\to\frac{\pi}{2}$. Assume by contradiction that there exists a subsequence $\varphi^n\to \frac{\pi}{2}$ such that $f_2(a_2,\varphi^n)\to b^2_\infty\in[0,+\infty)$ (resp. $f_1(a_1,\varphi^n)\to b^1_\infty\in\R$). The limit surface $\widetilde \Sigma_{\frac{\pi}{2}}(\infty, a_2,b_\infty)$ projects onto the  strip $\widetilde \Delta(\infty,a_2,\frac{\pi}{2})\subset\mathbb R^2$ (resp.  $\widetilde \Delta(a_1,\infty,\frac{\pi}{2})\subset\R^2$) and  it is a solution to the Jenkins-Serrin problem with boundary values $0$ over $l_1$, $b^2_\infty$ (resp. $b^1_\infty$) over $l_2$ and $+\infty$ over $l_3$. We may compare $\widetilde \Sigma_{\frac{\pi}{2}}(\infty, a_2,b^2_\infty)$ (resp.  $\widetilde \Sigma_{\frac{\pi}{2}}(a_1,\infty, b^1_\infty)$) with twice the fundamental piece of the helicoid $\mathcal H_{\infty,a_2}$ (resp. $\mathcal H_{a_1,\infty}$) see Section~\ref{subsec:relevant surfaces}, which is a vertical graph  (after a suitable ambient isometry) over $\widetilde \Delta(\infty,a_2,\frac{\pi}{2})$ (resp. $\widetilde \Delta(a_1,\infty,\frac{\pi}{2})$) with boundary values $b^1_\infty$ (resp. $-\infty$) over $l_2$,  $-\infty$ (resp. $0$) over $l_1$  and $+\infty$ over $l_3$. Twice the helicoid $\mathcal H_{\infty,a_2}$ (resp. $\mathcal H_{a_1,\infty}$)  has first-period function equals to $0$ along $l_2$ (resp. $l_1$) and it is below $\widetilde \Sigma_{\frac{\pi}{2}}(\infty, a_2,b^2_\infty)$ (resp. $\widetilde \Sigma_{\frac{\pi}{2}}(a_1,\infty, b^1_\infty)$) by looking at the boundary values and using the maximum principle. Then we get into a contradiction by the maximum principle because $\widetilde \Sigma_{\frac{\pi}{2}}(\infty, a_2,b^2_\infty)$ (resp. $\widetilde \Sigma_{\frac{\pi}{2}}(a_1,\infty,b^1_\infty)$) must have first period function equal to $0$.	
\end{proof}

\subsection{Solving the second period problem for the $(H,k)$-noids $\Sigma_\varphi(\infty,a_2,b)$}

\begin{lemma}\label{lemma-P2}
	Set $a_1=\infty$ and $(a_2,\varphi)\in\Omega_2$. In the notation of the second period problem, the following statements hold true:
	\begin{enumerate}
		\item $x(s)<0$ and $\pi<\psi(s)<2\pi$ for all $s\in (0,b)$.
		\item The curve $v_0$ intersects only once the geodesic $\gamma$.
		\item  If $\gamma$ intersects the $y$-axis with angle $\delta$, then $\varphi>\delta+2Hb$ and $\mathcal{P}_2^2(a_2,\varphi,b)=\cos(\delta)$.
		\item If $\mathcal P_2^2(a,\varphi,b)=\cos(\delta)$, then $\gamma$ intersects the $y$-axis with an angle $\delta$ and $\tfrac{y_0}{\sin(\psi_0)}>-\tfrac{1}{\sin(\delta)}$.
	\end{enumerate}
	
	Moreover:
	\begin{itemize}
		\item If $0<H<\tfrac 12$, $\lim_{a_2\to 0}\mathcal P_2^2(a_2,\varphi,f_2(a_2,\varphi))=\cos(\varphi)$ and  $\mathcal P_2^2(a_2,\varphi,f_2(a_2,\varphi))>1$ for $a_2$ close enough to $a_{\mathrm{max}}(\varphi)$.
		\item If $H=\tfrac 12$, $\lim_{a_2\to 0}\mathcal P_2^2(a_2,\varphi,f_2(a_2,\varphi))=\cos(\varphi)$ and  $\mathcal P_2^2(a_2,\varphi,f_2(a_2,\varphi))>1$ for $\varphi$ close enough to $\frac{\pi}{2}$.
	\end{itemize}

\end{lemma}

\begin{proof}
	We will identify $v_0$ with its projection in $\h^2$ in what follows.
	\begin{enumerate}
		\item 	As $v_0$ and $\pi(h_1)$ are orthogonal, $x(s) < 0$ in an interval $(0,\epsilon)$. Assume by contradiction that $x(s)<0$ is not true for all $s\in(0,b)$, then let $s_0$ be the first instant where $x(s_0)=0$. Let $U$ be the domain enclosed by the   arc of  $v_0$ joining $v_0(0)$ with  $v_0(s_0)$ and a segment in the $y$-axis joining $v_0(0)$ and $v_0(s_0)$. Let $\alpha$ be the non-oriented angle between $v_0$ and the $y$-axis at $v_0(s_0)$. Applying Gauss-Bonnet formula to the domain $U$ and taking into account that $\theta_0'>0$, we get the following contradiction
		\begin{align}
			0>-\text{area}(U)&=2\pi+\int_{0}^{s_0}k_g(s)ds-(\pi-\tfrac
			\pi 2+\pi-\alpha)\nonumber
			\\ &=\frac
			\pi 2+\alpha -\int_{0}^{s_0}\theta_0'(s)ds+2Hs_0
			>\tfrac{\pi}{2}-\varphi>0.\label{eq:G-B}
		\end{align}

		As $\theta_0'>0$, we know by Lemma~\ref{lemm:vertical} that the normal along $v_0$ points the interior of $\Delta$ and $k_g<2H$ with respect to the interior of $\Delta$. We have that $v_0$ stays locally in the concave side of the tangent curve of constant curvature $2H$ at $v_0(0)$. If $\psi(s)>\pi$ were not true for all $s\in(0,b)$, consider the first instant $s_0$ in which $\psi(s_0)=\pi$. At this point we have that $v_0$ has points locally around $v_0(s_0)$ in the mean convex side of the tangent curve of constant curvature $2H$ at $v_2(s_0)$, which contradicts the fact that $k_g<2H$.
		
		Assume again by contradiction that there is a first instant $s_0>0$ where $\psi(s_0)=2\pi$ and consider the domain $U$ enclosed by an arc of $v_0$ and a parallel segment to the $y$-axis at $v_0(s_0)$. Applying Gauss-Bonnet formula in $U$, we get the same contradiction as in Equation~\eqref{eq:G-B}.
		
		\item Assume once again by contradiction that $v_0$ intersects $\gamma$ twice, one in $v_0(s_0)$ and the other at $v_0(s_1)$ for some $0<s_0<s_1<b$. Then, the arc of $v_0$ joint with an arc of the curve $\gamma$ enclose a compact domain $U$. Applying Gauss-Bonnet formula to the domain $U$ we get the same contradiction as in Equation~\eqref{eq:G-B} because we have a similar region to the one given before. 
		
		\item As $\pi<\psi_0<2\pi$, we can parameterize the geodesic $\gamma$ as
		\begin{equation}\label{eqn:gamma}
			\gamma:(0,\pi)\to \mathbb{H}^2,\quad \gamma(t)=\left(x_0-y_0\dfrac{\cos(t)+\cos(\psi_0)}{\sin(\psi_0)}, -y_0 \dfrac{\sin(t)}{\sin(\psi_0)}\right).
		\end{equation}  
		
		\noindent As $\gamma$ intersects  the $y$-axis, the first coordinate of  $\gamma(0)$ is positive. Let $s_*$ be the instant where $\gamma$ intersects the $y$-axis, then  we can compute the oriented angle as
		\begin{equation}\label{eq:period2-angle}
			\cos(\delta)=\frac{\langle \gamma'(s_*), y\partial_y  \rangle }{|\gamma'(s_*)|}=\frac{x_0\sin(\psi_0)}{y_0}-\cos(\psi_0)=\mathcal P_2^2(a_2,\varphi,b).
		\end{equation}
		
		Consider now the domain $U$ enclosed by $v_0$, an arc of $\gamma$ and a segment of the $y$-axis. Applying Gauss-Bonnet formula, we have that
		\begin{align}
			0>-\text{area}(U)&=2\pi+\int_{0}^{b}k_g(s)ds-(\pi-\frac
			\pi 2+ \pi-\frac
			\pi 2+\pi-\delta)\nonumber
			\\ &=
			-\int_{0}^{b}\theta_0'(s)ds+2Hb+\delta
			=-\varphi+2Hb+\delta.\label{eq:G-B-2}
		\end{align}
		
		\item From \eqref{eq:second-period} and \eqref{eqn:gamma}, we get:
		\begin{align}\label{eq:gamma0}
			\gamma(\pi)=\left(y_0\frac{1+\mathcal P_2^2(a_2,\varphi,b)}{\sin(\psi_0)},0\right)\ \  \text{ and }\ \ \gamma(0)=\left( y_0\frac{\mathcal P_2^2(a_2,\varphi,b)-1}{\sin(\psi_0)},0\right)
		\end{align}
		
		\noindent whose first coordinates are negative and  positive, respectively. That means that $\gamma$ intersect the $y$-axis, whence the intersection angle is $\delta$ by~\eqref{eq:period2-angle}.	
		
		On the other hand, as $v_0$ only intersects $\gamma$ once, we deduce that the second coordinate of $\gamma(s_*)$ (the instant where $\gamma$ intersects the $y$-axis) is less than $1$. Then  we get that 
		$-y_0\tfrac{\sin(\delta)}{\sin(\psi_0)}<1,$
		and the inequality of the statement follows.
	\end{enumerate}
	
	Let us compute the limits. Integrating $\psi'=\theta'-\cos(\psi)-2H$ along $v_0$, see formula~\eqref{eq:psi}, and taking into account that $\psi(b)-\psi(0)=\psi_0- \pi$ and $\theta_0(b)-\theta_0(0)=\varphi$, we obtain
	\begin{equation}
		\psi_0=\varphi +\pi-\int_{0}^b\cos(\psi(s))ds-2H b.
	\end{equation}
	In particular,   when $b\to0$, we have that $\psi_0$ converges to $\varphi+\pi$ and consequently $\mathcal P_2^2(a_2,\varphi,b)$ converges to $\cos(\varphi)$. Thus, if $a_2^n$ is a sequence converging to $0$, then $f_2(a_2^n,\varphi)$ also converges to $0$, so that we obtain the desired limit.
	
Assume now that $0<H<\tfrac 1 2$.	Let us consider a sequence $a_2^n\to a_\mathrm{max}(\varphi)$. 
	If we translate properly vertically the surface $\widetilde\Sigma_\varphi(\infty,a_2^n,f_2(a_2^n,\varphi))$ for it takes the value $0$ along $l_2$, we have that $\widetilde\Sigma_\varphi(\infty,a_2^n,f_2(a_2^n,\varphi))$ converges to twice the fundamental piece of the conjugate surface of the embedded $H$-catenoids constructed in~\cite{CMR,Plehnert2}. Here we are using that $f_2(a_2^n,\varphi)\to+\infty$ by Lemma~\ref{lemma:P1}. However, as in the setting of the second period problem, we are translating and rotating $\Sigma_\varphi(\infty,a_2,f_2(a_2^n,\varphi))$ so that $v_0^n(0)=(0,1,0)$ and $(v_0^n)'(0)=-E_1$, that means  $\widetilde \Sigma_\varphi(\infty,a_2,f_2(a_2^n,\varphi))$ converges to a subset of a vertical plane in this setting. We obtain that the limit surface of $\Sigma_\varphi(\infty,a_2,f_2(a_2^n,\varphi))$ is not twice the fundamental piece of the $H$-catenoid but a subset of the vertical $H$-cylinder  that projects onto a curve of constant curvature $2H$  orthogonal to the $y$-axis at $(0,1)$. The $H$-cylinder can  be parameterized as $\alpha\times\R$ with $\alpha:(-\arccos(2H),\arccos(2H) )\to \h^2$ given by
	\[\alpha(s)=\frac{1}{1-2H}(\sin(s),-2H+\cos(s)).\]  We have that $x_0^n\to \frac{-1-2H}{\sqrt{1-4H^2}}<0$ and $y_0^n\to 0$. Moreover, for large $n$, the curve $\gamma^n$ does not intersect the $y$-axis since we have shown that the limit is a subset of the $H$-cylinder $\alpha\times\R$. That means that the first coordinate of $\gamma^n(0)$ is positive and $\mathcal P_2^2(a_2^n,\varphi,f_2(a_2^n,\varphi))>1$ since we have proved that $\sin(\psi_0^n)<0$, see Equation~\eqref{eq:gamma0}. 

Assume now that $H=\tfrac 12$ and consider  the sequence $\widetilde \Sigma_{\varphi_n}(\infty, a_2, f_2(a_2,\varphi^n))$ with $\varphi^n\to \frac{\pi}{2}$.  By Lemma~\ref{lemma:P1}, after a suitable translation, the limit surface is twice the fundamental piece of the helicoid $\mathcal H_{\infty,a_2}$. The conjugate limit surface  $\Sigma_{\frac \pi2}(\infty, a_2,+\infty)$ is an embedded $\frac 1 2$-catenoid constructed in~\cite{CMR,DH,Plehnert}. However, in the setting of the second period problem, that is,  $v_0^n(0) = (0,1,0)$ and $(v^n_0 )'(0) = -E_1$, the conjugate limit surface is not a $\frac 1 2$-catenoid but a subset of the horocylinder $\{y=1\}$.

Since the family $\widetilde\Sigma_\varphi(a_2, f_2(a_2,\varphi))$ is continuous in the parameters $a_2$ and $\varphi$, so is the conjugate family. We have that $v_0^n$ converges to the line $\{y=1, z=0\}$. Then we have that $x_0^n\to+\infty$ and $y_0^n\to 1$.  Therefore, as $\sin(\psi_0^n) <0$, we deduce that $\mathcal{P}^2_2(a_2,\varphi^n,f(a_2,\varphi^n))>1$ for $n$ large enough.			
\end{proof}

\begin{lemma}\label{lemma:vertical graph}
	The surface $\Sigma_\varphi(\infty,a_2,b)$ is a vertical graph. In particular, it is embedded.
\end{lemma}

\begin{proof}
	We continue working in the half-space model and the setting of the second period problem. First observe that $v_0$ and $v_2$ are embedded curves since $\theta_i'>0$ and $\int_{\widetilde v_i}\theta_i'\leq\pi$ $i=0,2$ (see~\cite[Lemma 2.1]{CMR}). In particular each curve of the boundary of $\Sigma_\varphi(\infty,a_2,b)$ is embedded.  We will  show that $\partial \Sigma_\varphi(\infty,a_2,b)$ projects one-to-one to a curve of $\h^2$, so that $\Sigma_\varphi(\infty,a_2,b)$ is a graph by a standard application of the maximum principle.

	Assume that $0<H<\tfrac 1 2$. Observe that the curves $\pi(v_1)$ and $\pi(h_1)$ do not intersect each other, since they are consecutive and $\pi(v_1)$ is a constant curve of curvature $2H$ and $\pi(h_1)$ is contained in a orthogonal geodesic to $\pi(v_1)$. Moreover, by item (1)  in Lemma~\ref{lemma-P2}, the curve $\pi(v_0)$ does not intersect $\pi(h_1)$ or $\pi(v_1)$.
	If  $\pi(h_2)\subset\gamma$ does not intersect in the interior of $\pi(v_1)$, we are done  (as  $\Sigma_\varphi(\infty,a_2,b)$ is a multigraph then $\pi(v_2)$ cannot intersect any component of $\pi(\partial\Sigma_\varphi(\infty,a_2,b))$ and then we conclude that the multigraph is a graph).
	Otherwise, if $\pi(h_2)\subset\gamma$  intersects $\pi(v_1)$, then  as $\pi(v_2)$ must join these two curves enclosing a multigraph, we  would obtain that $\pi(v_2)$ intersects itself. However, this is not possible  and we conclude that $\partial \Sigma_\varphi(\infty,a_2,b)$ projects one-to-one.
	
	The proof for $H=\frac{1 }{2}$ is similar, yet $v_1$ does not exist and the curve $\pi(h_3)$ is not compact.
\end{proof}

\begin{remark}
For $H=0$, the embeddedness of the fundamental piece $\Sigma_\varphi(\infty,a_2,b)$ is guaranteed by the Krust property, that is, the conjugate surface of a minimal graph in $\h^2\times\mathbb R$ over a convex domain  is a minimal graph in $\h^2\times\mathbb R$, see~\cite[Theorem 14]{HST}. However, for $H>0$ there is not a Krust property (see~\cite{CMR}) and the embeddedness has to be proven in order to show the global embeddedness of the $(H,k)$-noids with genus one for some values of $H$, see Proposition~\ref{Theorem:embebimiento}.
\end{remark}

\begin{theorem}\label{th:k-noides}
	For each $k\geq3$ and $\frac{\pi}{k}<\phi\leq\frac{\pi}{2}$, there exists  a properly  Alexandrov embedded $H$-surface with $0\leq H\leq \frac{1}{2}$  in $\h^2\times\R$ with genus $1$ and $k$ ends.  These $H$-surfaces have dihedral symmetry with respect to $k$ vertical planes and they are symmetric with respect to a horizontal plane. Moreover, if $0<H<\tfrac 12 $ each of their ends is embedded and asymptotic to (and contained in the concave side of) a vertical $H$-cylinder.
\end{theorem}
\begin{proof}
	The case $H=0$ is treated in~\cite{CM}. 
	Assume first that $0<H<\tfrac 12$ and take $k\geq 3$ and $\phi\in (\frac{\pi}{k}, \frac{\pi}{2})$. We choose $\varphi=\phi$ and,
	by Lemma~\ref{lemma-P2},	we have that $\mathcal P_2^2(a_2,\varphi,f_2(a_2,\varphi))$ tends to $\cos(\varphi)$ when $a_2\to0$ and becomes greater than $1$  when $a_2\to a_{\mathrm{max}}(\varphi)$. By the continuity of $\mathcal P_2^2$, there exists $a_\varphi$ such that  $\mathcal P_2^2(a_\varphi,\varphi,f_2(a_\varphi,\varphi))=\cos(\frac{\pi}{k})$. Therefore, the surface $\Sigma_\varphi:=\Sigma_\varphi(\infty, a_\varphi,f_2(a_\varphi,\varphi))$ solves the two period problems, so after successive reflections over the vertical planes and the horizontal plane of symmetry, we obtain the desired complete $H$-surface with genus $1$ and $k$ ends asymptotic to vertical $H$-cylinders from the concave side. 
	We shall see now that the ends are embedded. First, observe that, by the maximum principle with respect to horizontal planes arriving from above, $\Sigma_\varphi$ is contained in the slab $\h^2\times(-\infty,0]$ (we are assuming after a vertical translation that $v_0$ and $v_2$ lies in $\h^2\times\{0\})$. Moreover, if we reflect $\widetilde \Sigma_\varphi$ about the horizontal geodesic $\widetilde h_1$, the total variation of the angle of rotation $\theta_0$ along the complete vertical line $\widetilde v_0^*$ of the reflected surface $\widetilde\Sigma^*_\varphi$ is $2\varphi<\pi$, whence the curve $v_0^*$ (the extension of the curve $v_0$ after reflection) is embedded by \cite[Lemma 2.1]{CMR}. Therefore, the conjugate surface of the reflected surface  $\widetilde\Sigma_\varphi^*$ is a vertical graph contained in the  half-space $\h^2\times(-\infty,0)$. Then, after reflecting over the horizontal plane $\h^2\times\{0\}$, we obtain an embedded surface that contains an end and this proves that the ends are embedded.  
	
	Assume now that $H=\tfrac 1 2$ and  take $k\geq 3$ and the parameter $\phi\in (\frac{\pi}{k}, \frac{\pi}{2})$. Again we will use a continuity argument to prove that there exist  parameters $(a(\phi), \varphi(\phi))\in\Omega_2$ that solve both period problems.  We define the foliation of $\Omega_2$ by the family of curves $\{\alpha_{\phi}:[0,1]\to \Omega: \phi\in(0,\frac{\pi}{2})\}$ where
	\begin{equation}\label{eq:foliation}
		\alpha_{\phi} (t) = (1-t) (0,\phi) + t (\tan(\tfrac{\pi}{2}-\phi),\tfrac{\pi}{2}). 
	\end{equation}
	By Lemma~\ref{lemma-P2}, we get
	\[\lim_{t\to 0} \mathcal{P}_2^2(\alpha_{\phi}(t), f_2(\alpha_\phi(t)))=\mathcal{P}_2^2(0,\phi,f_2(0,\phi)) = \cos(\phi).\] 
	
	For $t_\epsilon=1-\epsilon$ with $\epsilon>0$ small enough we have that, the second coordinate of $\alpha_\phi$ is $\frac \pi 2 -\epsilon(\frac \pi 2-\phi)$. Hence, by Lemma~\ref{lemma-P2}, we get $\mathcal{P}^2_2(\alpha_{\phi}(t_\epsilon), f(\alpha_{\phi}(t_\epsilon)))>1$.
	Since $\cos(\phi) < \cos(\frac{\pi}{k})$, by continuity, there exists an instant $t_*\in (0,1)$ such that $\mathcal{P}^2_2(\alpha_{\phi}(t_*), f_2(\alpha_{\phi}(t_*)))=\cos(\frac{\pi}{k})$. We have proved that, for each $\phi$, there exists at least  $(a(\phi), \varphi(\phi))=\alpha_\phi(t_*)$ such that $\Sigma_\phi=\Sigma_{\varphi(\phi)}(\infty,a(\phi),f_2(a(\phi),\varphi(\phi)))$ solves both periods problems. Then,  after successive reflections over the vertical planes and the horizontal plane of symmetry, we obtain a complete $\frac{1}{2}$-surface with genus $1$ and $k$ ends.
\end{proof}

\begin{remark}\label{rm:m k}
	We also obtain $H$-surfaces with genus $1$ and $k\geq 5$ ends when the first period function vanishes and the second period function is equal to $\cos(\frac{m \pi}{k})$ with $m<\frac k 2$ and $\text{gcd}(m,k)=1$. If $m>1$, the  $H$-surfaces constructed close after $m$ laps around the origin and they are never embedded, see Figure~\ref{Fig-knoid} (right).
	
\end{remark}

\begin{figure}[htb]
	\begin{center}
		\includegraphics[height=5cm]{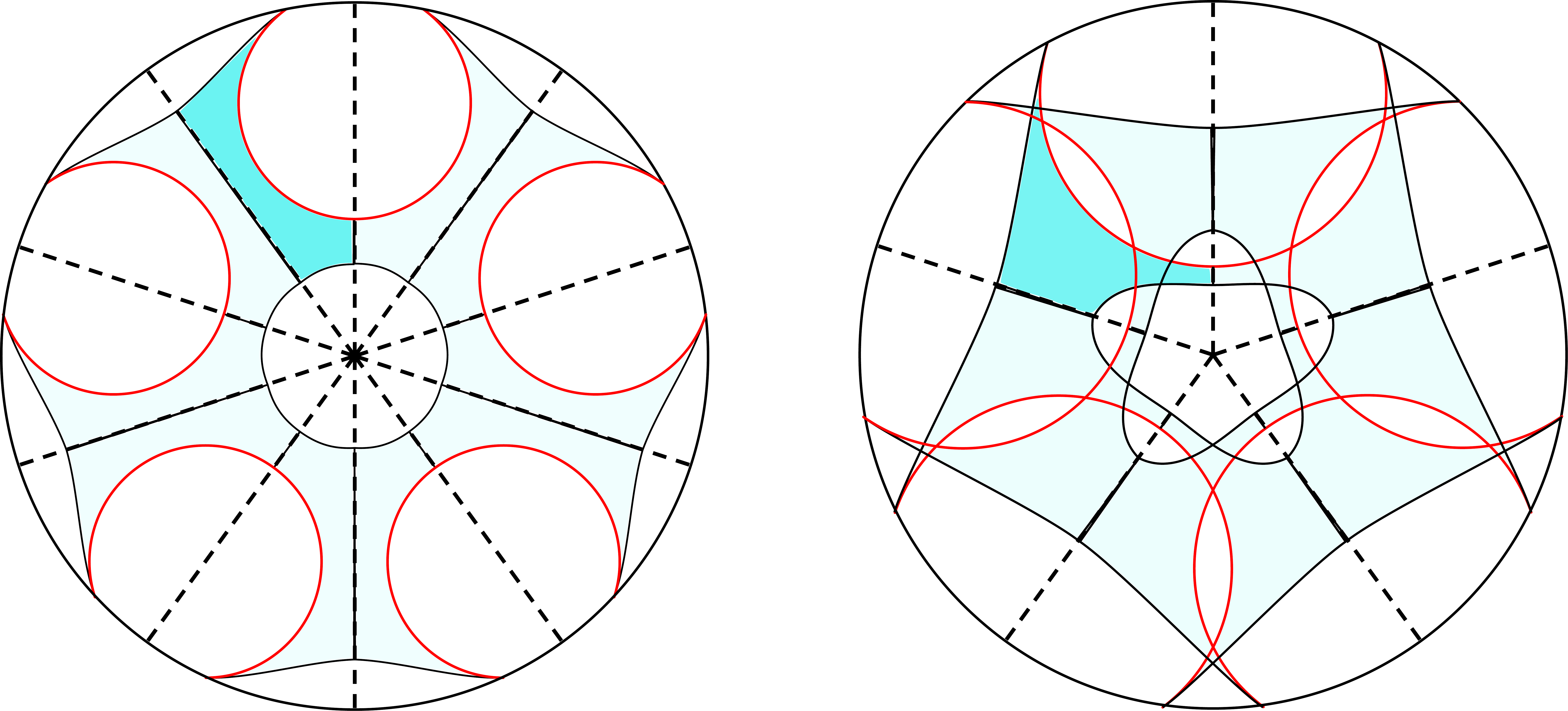}
	\end{center}
	\caption{The projection onto $\h^2$ in the disk model of a $(H,5)$-noid with $\mathcal P_2^2=\cos(\frac{\pi}{5})$ (left) and $(H,5)$-noid with $\mathcal P_2^2=\cos(\frac{2\pi}{5})$.}
	\label{Fig-knoid}
\end{figure}

\begin{proposition}\label{Theorem:embebimiento} The $(H,k)$-noids with genus one  given by Theorem~\ref{th:k-noides} are embedded for $\frac{1}{2}\cos(\frac{\pi}{k})<H \leq \tfrac 12$. In particular, for $\frac{1}{4}<H\leq\frac{1}{2}$, all $(H,3)$-noids with genus one are embedded.
\end{proposition}
\begin{proof}
	
	Observe that the embeddedness of each $(H,k)$-noid with genus one can be guaranteed if the extended surface of $\Sigma_\varphi$ by the reflection about $h_2$ is embedded, or equivalently if the extension of the curve $v_2$ is embedded after the reflection. As $\Sigma_\varphi$ is a vertical graph, this is equivalent to the fact that $v_2$ intersects only once the geodesic~$\gamma$ ($v_2$ has not self-intersections after the reflection over $h_2$.)

		Assume first that $\frac{1}{2}\cos(\frac{\pi}{k})<H<\tfrac 1 2$.
		Consider $p=\pi(h_3)=(p_1,0)$, which coincides with the ideal endpoint of $v_1$ and $v_2$; in particular, the first coordinate of $p$ verifies $p_1<\frac{-1-2H}{\sqrt{1-4H^2}}$. Moreover, we will see  that $v_2$  intersects $\gamma$ just once if and only if $p_1$ is smaller than the first coordinate of $\gamma(\pi)$, which will be denoted by $\gamma(\pi)_x$. Assume by contradiction that $p_1$ is smaller than the first coordinate of $\gamma(\pi)$ and $v_2$ intersect $\gamma$ at least twice. Then there exists a compact  arc of $v_2$ with extremes in $\gamma$ contained in the region of $\h^2\backslash \gamma$ which does not contain $\pi(v_1)$, however along this arc the curvature can not satisfies $k_g<2H$ with respect to the normal that points to the interior of $\Delta$. Then,
		 using the inequality in Lemma~\ref{lemma-P2} item (4) and Equation~\eqref{eq:gamma0}, if $\mathcal P_2^2=\cos(\frac{\pi}{k})$ we have    $\gamma(\pi)_x>-\frac{1+\cos(\frac{\pi}{k})}{\sin(\frac{\pi}{k})}$. Then, for $H>\frac{1}{2}\cos(\frac{\pi}{k})$ we have that
		\[p_1<\frac{-1-2H}{\sqrt{1-4H^2}}<-\frac{1+\cos(\frac{\pi}{k})}{\sin(\frac{\pi}{k})}<\gamma(\pi)_x,\]
		which proves the case $H<\tfrac 12$.

	Assume now that $H=\tfrac 1 2$.  We will prove that the first coordinate of the curve $v_2$ goes to $-\infty$. This means that $v_2$ can intersect $\gamma$ only once as in the previous case, so the surface will be embedded.  Consider $\widetilde\Sigma^n_\phi:=\widetilde\Sigma^n_{\varphi(\phi)}(n,a(\phi),f_2(a(\phi),\phi))$ the sequence of minimal graphs over $\widetilde\Delta_n$ converging to $\widetilde\Sigma_\phi$ and its respective conjugate surfaces $\Sigma_\phi^n$ converging to $\Sigma_\phi$. On the one hand, let $\widetilde v_1^n\subset\partial\widetilde\Sigma_\phi^n$ and $\widetilde v_2^n\subset\partial\widetilde\Sigma_\phi^n$ be the vertical geodesics projecting onto $p_1^n$ and $p_2^n$ respectively, and let $v_1^n$ and $v_2^n$ their conjugate curves contained in  horizontal planes.  Let $k_g^n=1-(\theta_1^n)'$ be the curvature of $v_1^n$ with respect to the normal that points to the interior of the domain $\Delta^n$ where $\Sigma_\phi^n$ is projecting. We know that $k_g^n$ approaches $1$ as $n\to\infty$. On the other hand, the second coordinate of $h_1^n$ diverges since we have shown that $\pi(h_1)$ is not compact. Then we have that the two coordinates of $v_1^n$ diverge. In particular, the first coordinate of $v_2^n$ also diverges to $-\infty$ and the embeddedness follows.
\end{proof}
\begin{remark}\label{REMARK-Emb}
	We can also guarantee the embeddedness of the $(H,k)$-noids for $0<H<\tfrac 1 2$ if the value $a_\varphi$ verifies $a_\varphi>a_{\mathrm{emb}}$, see also~\cite{CM}. This condition means that the angle at the point $p_2$ is less or equal than $\frac{\pi}{2}$. Let $\theta_2^*$ be the angle of rotation of ${\widetilde v_2}^*$, the extension of $\widetilde v_2$. As $\int_{{\widetilde v_2}^*}{\theta_2}^*<\pi$ then ${v_2}^*$, the extension of $v_2$ by  reflection, is embedded by \cite[Lemma 2.1]{CMR} and therefore the $(H,k)$-noid is also embedded. 
\end{remark}

\subsection{$H$-surfaces with infinitely many ends}\label{subsec:infty-ends}
We are going to analyze now the case where the first problem is solved but $\mathcal P_2\geq 1$. 

\begin{itemize}
	\item 	When $\mathcal P_2^2(a_2,\varphi,f_2(a_2,\varphi))=1$,  $h_1$ and $h_2$ lies in vertical asymptotic planes, so after successive reflections over the vertical planes and the horizontal plane, we obtain a periodic surface invariant by a discrete group of parabolic translations that fix the common vertical ideal line of the vertical planes of symmetry. This gives us a $1$-parameter family of  parabolic $(H,\infty)$-noids with one limit end that is, the ends are accumulating in an ideal vertical geodesic. By similar arguments to those of~Proposition~\ref{Theorem:embebimiento} one can prove that they are embedded for $H=\tfrac 1 2$.   
	
	\item When $\mathcal P_2(a_2,\varphi,f_2(a_2,\varphi))>1$,  $h_1$ and $h_2$ lie in two disjoint vertical  planes, so after successive reflections over the vertical planes and the horizontal plane, we obtain a periodic surface invariant by a discrete group of hyperbolic translations spanned by successive reflections over the vertical planes. This gives us a $2$-parameter family of  hyperbolic $(H,\infty)$-noids with two limit ends, that is, the ends are accumulating in two different ideal vertical geodesic. Similar arguments to  the proof of Theorem~\ref{th:k-noides} show that the ends of these surface are embedded. Moreover  in this case for $0<H<\tfrac 1 2$  we have more freedom and we can choose  $a_2>a_\mathrm{emb}(\varphi)$,whence the reflected surface of $\Sigma_\varphi(\infty,a_2,f_2(a_2,\varphi))$ about the vertical plane containing $h_2$ is embedded, and consequently the complete surface is embedded. For $H=\tfrac 1 2 $ they are always embedded by similar arguments to those of~Proposition~\ref{Theorem:embebimiento}.  
\end{itemize}

	We state the following result:
\begin{theorem}\label{th:planar-domain-sub}
There exists properly embedded $H$-surfaces in $\h^2\times\R$ with genus zero, infinitely many ends and two limit ends for $0\leq H\leq \tfrac 1 2$.
\end{theorem}
\begin{remark}
Properly embedded surfaces with genus zero, and a finite number of ends were constructed in~\cite{CMR}. Observe that in the case of $H=\tfrac 12$, the parabolic $(H,\infty)$-noids are properly embedded surfaces with infinitely many ends and one limit end.
\end{remark}

\subsection{Solving the second period problem for the $(H,k)$-nodoids $\Sigma_\varphi(a_1,\infty,b)$}

\begin{lemma}\label{lemma-P2-2}
	Set $a_2=\infty$ and $(a_1,\varphi)\in\Omega_1$. Under the  assumptions of the second period problem, the following statements hold true:
	
	\begin{enumerate}
		\item If $|\mathcal P^1_2(a_1,\varphi,b)|<1$, then $\gamma$ intersects the $y$-axis with an angle $\delta$ with \[\cos(\delta)=\begin{cases}
			\mathcal P^1_2(a_1,\varphi,b),\   \text{if } \sin(\psi_0)<0,\  \\
			-\mathcal P^1_2(a_1,\varphi,b),  \text{if } \sin(\psi_0)>0.
		\end{cases}\]

		\item Assume $0< b<b_\mathcal I^1(\infty,\varphi)$. We have that
		\begin{itemize}
			\item  $x(s)<0$ and $\psi(s)\in(\pi,2\pi)$;
			\item   if $\gamma$ intersect the $y$-axis with an angle $\delta$, then $\varphi>\delta+2H b$ and $\mathcal P^1_2(a_1,\varphi,b)>0$.

		\end{itemize}
		\item If $b<0$ then $\psi(s)>\pi$. 
		
	\end{enumerate}

	Moreover:
	\begin{itemize} 
		\item If $0<H<\frac{1}{2}$, $\lim_{a_1\to 0}\mathcal P^1_2(a_1,\varphi,f_1(a_1,\varphi))=\cos(\varphi)$ and $|\mathcal P^1_2(a_1,\varphi,f_1(a_1,\varphi))|>1$ for $a_1$ close to $a_{\mathrm{max}}(\varphi)$. In fact, there exist $0<\varphi_-<\varphi_+<\frac \pi 2$ such that   $\mathcal P^1_2(a_1,\varphi,f_1(a_1,\varphi))>1$ for all $\varphi<\varphi_-$ and $a_1$ close to $a_{\mathrm{max}}(\varphi)$, and  $\mathcal P^1_2(a_1,\varphi,f_1(a_1,\varphi))<-1$ for all $\varphi>\varphi_+$ and $a_1$ close to $a_{\mathrm{max}}(\varphi)$.
	\item If $H =\tfrac 12$,  $\lim_{a_1\to 0}\mathcal P^1_2(a_1,\varphi,f_1(a_1,\varphi))=\cos(\varphi)$ and $\mathcal P^1_2(a_1,\varphi,f_1(a_1,\varphi))<-1$ for $\varphi$ close enough to $\frac{\pi}{2}$.
	\end{itemize}
	
\end{lemma}
\begin{proof} In what follows, we consider the surfaces $\widetilde\Sigma^-_\varphi(a_1,\infty,b)$ and $\Sigma^-_\varphi(a_1,\infty,b)$ its conjugate  $H$-surface as in the setting of the second period problem.
	\begin{enumerate}
		
		\item Observe that, if $\sin(\psi_0)=0$, then $|\mathcal P^1_2(a_1,\varphi,b)|=1$. Assume first that  $\sin(\psi_0)<0$ and proceed as in Lemma~\ref{lemma-P2}. We  parameterize the curve $\gamma$ as in Equation~\eqref{eqn:gamma}. The same computation of item~(3) in Lemma~\ref{lemma-P2} tells us that $\gamma$ intersects the $y$-axis with an angle $\delta$ given by Equation~\eqref{eq:period2-angle}.
		
		Otherwise, if $\sin(\psi_0)>0$, we parameterize the curve $\gamma$ as
		\begin{equation}\label{eqn:gamma2}
			\gamma:(0,\pi)\to \mathbb{H}^2,\quad \gamma(t)=\left(x_0-y_0\dfrac{\cos(\pi - t)+\cos(\psi_0)}{\sin(\psi_0)}, y_0 \dfrac{\sin(\pi - t)}{\sin(\psi_0)}\right).
		\end{equation}
		We get
		\begin{align}\label{eq:gamma0-2}
			\gamma(0)=\left(y_0\frac{1+\mathcal P^1_2(a_1,\varphi,b)}{\sin(\psi_0)},0\right)\ \ \text{and}\ \ \gamma(\pi)=\left( y_0\frac{-1+\mathcal{P}^1_2(a_1,\varphi,b)}{\sin(\psi_0)},0\right),
		\end{align}  
		whose first coordinates are positive and negative, respectively. That means that $\gamma$ intersects the $y$-axis. The angle of intersection $\delta$ at the instant $s^*$ where $\gamma$ intersects the $y$-axis satisfies
		\begin{equation}\label{eq:period2-angle2}
			\cos(\delta)=\frac{\langle \gamma'(s_*), y\partial_y  \rangle }{|\gamma'(s_*)|}=-\frac{x_0\sin(\psi_0)}{y_0}+\cos(\psi_0)=-\mathcal P_2^1(a_1,\varphi,b).
		\end{equation}

		\item As the angle of rotation along $\widetilde v_0^-$ turns in a positive sense   ($\theta_0'>0$), we can apply item (1), (2) and (3) of Lemma~\ref{lemma-P2} obtaining similar results. Assume now by contradiction that $\mathcal P^1_2(a_1,\varphi,b)<0$. By Equation~\eqref{eq:second-period} we get that
		\[x_0=y_0\frac{\mathcal{P}^1_2(a_1,\varphi,b)+\cos(\psi_0)}{\sin(\psi_0)}.\]
		If $\mathcal P^1_2(a_1,\varphi,b)\leq-1$, we obtain that $x_0\geq0$ which is a contradiction. If  $-1<\mathcal P^1_2(a_1,\varphi,b)\leq0$, by item~(1) we have that $\delta=\arccos(\mathcal P^1_2(a_1,\varphi,b))\geq\frac\pi 2>\varphi$ which contradicts  $\varphi>\delta+2H|b|$.                                  
		
		\item 	As $\theta_2'<0$, we know by Lemma~\ref{lemm:vertical} that the normal along $v_0^-$ points to the exterior of $\Delta$ and $k_g>2H$ with respect to this normal. We have that $v_0^-$ stays locally in the mean convex side of the tangent curve of constant curvature $2H$ at $v_0^-(0)$. If $\psi(s)>\pi$ were not true for all $s\in(0,b)$, let us consider the first instant $s_0>0$ in which $\psi(s_0)=\pi$. At this instant, we have that $v_0^-$ contains points locally around $v_0^-(s_0)$ in the non-mean convex side of the tangent curve of constant curvature $2H$ at $v_0^-(s_0)$, which contradicts $k_g>2H$.
	\end{enumerate}	
	
	Let us now analyze the limits. Assume that $b<0$. Integrating along $v_0^-$ the identity $\psi'=-\theta'-\cos(\psi)+2H$ (see Formula~\ref{eq:psi} and Remark~\ref{Remark:signo}) and taking into account that here $\psi(b)-\psi(0)=\psi_0- \pi$ and $\theta_0(|b|)-\theta_0(0)=-\varphi$ since $\theta_0'<0$, we obtain
	\begin{equation}
		\psi_0=\varphi +\pi-\int_{0}^{|b|}\cos(\psi(s))ds+2H |b|.
	\end{equation}
	In particular,   when $b\to0^-$, we have that $\psi_0$ converges to $\varphi+\pi$ and consequently $\mathcal P^1_2(a_1,\varphi,b)$ converges to $\cos(\varphi)$ as $b\to 0$. If we take a sequence $a_1^n$ converging to $0$, then Lemma~\ref{lemma:P1} implies that $f_1(a_1^n,\varphi)\to0^-$ and we get that $\lim_{a_1\to 0}\mathcal P^1_2(a_1,\varphi,f_1(a_1,\varphi))=\cos(\varphi)$.
	
Assume first that $H<\tfrac 1 2$.	Let us consider a sequence $a_1^n\to a_\mathrm{max}(\varphi)$. By Lemma~\ref{lemma:P1} we get that $f_1(a_1^n,\varphi)\to-\infty$, so  $\widetilde\Sigma_\varphi(a_1^n,\infty,f_1(a_1^n,\varphi))$ converges to twice the fundamental piece of the conjugate surface of the $H$-catenodoids constructed in~\cite{CMR} and therefore $\Sigma_\varphi^-(a_1^n,\infty,f_1(a_1^n,\varphi))$ converges to twice the fundamental piece of an $H$-catenodoid. Nevertheless, as in the setting of the second period problem we are translating and rotating $\Sigma_\varphi^-(a_1^n,\infty,f_1(a_1^n,\varphi))$ in order to $v_0^{n-}(0)=(0,1,0)$ and ${(v_0^{n-})'}(0)=-E_1$. We obtain that the limit surface is not twice the fundamental piece of the $H$-catenoid but a subset of the $H$-cylinder  that projects onto a curve of constant curvature $2H$  orthogonal to the $y$-axis at $(0,1)$. The $H$-cylinder can   be parameterized as $\alpha\times\R$ with $\alpha:(-\arccos(-2H),\arccos(-2H) )\to \h^2$ given by
	\[\alpha(s)=\frac{1}{1+2H}(\sin(s),2H+\cos(s)).\]  We deduce that $x_0^n\to \frac{-1+2H}{\sqrt{1-4H^2}}<0$ and $y_0^n\to 0$.
	
	To understand the limit we distinguish if the limit  (after translation) $H$-catenodoid is or is not embedded.
	
	Assume first that the limit $H$-catenodoid is not embedded. We translate and rotate the surface $\Sigma_\varphi^-(a_1^n,\infty,f_1(a_1^n,\varphi))$ such that (in the half-space model) $(h_1^-)^n$ is contained in the vertical plane $\{x=0\}$ and $v_0^{n-}$ and $v_2^{n-}$ are contained in the horizontal  plane $\h^2\times\{0\}$. In the limit  we have that the projection of $(v_1^-)^\infty\in\h^2\times\{+\infty\}$ intersects twice the geodesic $\{x=0\}\subset\h^2$ and the same happens for the curves $(v_0^-)^\infty$ and $(v_2^-)^\infty$, see Figure~\ref{Fig-Lim-Inf} ((A) - up left). By the continuity of the conjugation (see~\cite[Proposition 3.3]{CMT}), the same happens for the curves $v_i^{n-}$ with $n$ large enough, see Figure~\ref{Fig-Lim-Inf} ((A) - up right). Then we rotate the surface $\Sigma_\varphi^-(a_1^n,\infty,f_1(a_1^n,\varphi))$ until $(h_2^-)^n$ lies in the vertical plane $\{x=0\}$ and $v_0^{n-}(0)=(0,1,0)$ and $(v_0^{n-})'(0)=-E_1$ (the setting of the second period problem). We have that the projections of $v_0^{n-}$, $v_1^{n-}$ and $v_2^{n-}$ intersect twice the vertical plane containing the curve $h_1^{n-}$, see~Figure~\ref{Fig-Lim-Inf} ((A) - bottom right). In particular, we have that $\psi_0^n\in(2\pi,3\pi)$, that is, $\sin(\psi_0^n)>0$. Moreover, the curve $\gamma^n$ intersects twice the curve $v_1^{n-}$ and in particular $\gamma^n$ does not intersect the $y$-axis, then  by Equation~\eqref{eq:gamma0-2} we deduce that $\mathcal P_2^1(a_1^n,\varphi,f_1(a_1^n,\varphi))<-1$ for $n$ large enough since  $\sin(\psi_0^n)>0$.
	
	\begin{figure}[htb]
			\includegraphics[height=6.5cm]{FIG-LIMITEINF-JUNTOS.pdf}
			\caption{\label{Fig-Lim-Inf}}
		\subcaption{The projection of the surfaces $\Sigma^-_\varphi(a_{\mathrm{max}},\infty,-\infty)$ (left) and $\Sigma^-_\varphi(a_1^n,\infty,f_1(a_1^n,\varphi))$ for large $n$ (right) with $0<H<\tfrac 1 2$.}
		 \subcaption{The projection of the surfaces $\Sigma^-_{\pi/2}(a_1,\infty,-\infty)$ (left) and $\Sigma^-_{\varphi^n}(a_1,\infty,f_1(a,\varphi^n))$ for large $n$ (right) with $H=\tfrac 1 2$.\newline
		  Up,  we are assuming the curve $h_1^-$ is contained in the vertical plane $\{x=0\}$. Down,  we are assuming that $h_2^-$ is contained in the vertical plane $\{x=0\}$ (the setting of the second period problem). The projections of the vertical planes of symmetry are represented in dashed lines.}
	
	\end{figure}
	
	If the limit $H$-catenodoid is embedded (not in the boundary case where $(v_0^-)^\infty$ and its reflected copy intersect each other  in a point of the asymptotic boundary), the argument is analogous but, in this case, by the continuity of conjugation the curves     $\pi(v_0^{n-})$,  $\pi(v_1^{n-})$ and $\pi(v_2^{n-})$ intersect only  once $\gamma^n$ (the projection of the vertical plane containing the curve $h_1^{n-}$) for large $n$ obtaining that $\psi_0^n\in(\pi,2\pi)$. A similar analysis shows that in this case $\mathcal P_2^1(a_1^n,\varphi,f_1(a_1^n,\varphi))>1$ for large $n$. 
	
	On the other hand, if $\varphi\to\frac\pi 2$, we have that $a_{\mathrm{max}}(\varphi)\to0$, and then~\cite[Proposition 4.8]{CMR} ensures that the limit $H$-catenodoid is not embedded. If  $\varphi\to 0$, we have that $a_{\mathrm{max}}(\varphi)\to+\infty$, and then~\cite[Proposition 4.8]{CMR} ensures that the limit $H$-catenodoid is embedded, which completes the case $H<\tfrac 1 2 $.
	
	Assume now that $H=\tfrac{1}{2}$. Let us consider a sequence $\varphi ^n\to \frac\pi2$. By Lemma \ref{lemma:P1} we have that $f_1(a_1,\varphi^n)\to-\infty$, therefore $\widetilde\Sigma_{\varphi^n}(a_1,\infty,f_1(a_1,\varphi^n))$ converges to twice the fundamental piece of the helicoid $\mathcal H_{a_1,\infty}$ of Section~\ref{subsec:relevant surfaces}.
 The conjugate surface $\Sigma_{\frac \pi 2}^-(a_1,\infty,-\infty)$ is twice the fundamental piece of a non-embedded $\frac 1 2$-catenodoid, see~\cite[Section 4.3]{CMR}.  Nevertheless, as in the setting of the second period problem, we are translating and rotating  $\Sigma_{\varphi^n}^-(a_1,\infty,f(a_1,\varphi^n))$ in order to have $v_0^{n-}(0)=(0,1,0)$ and $(v_0^{n-})'(0)=-E_1$.
   We obtain that the limit surface is not twice the fundamental piece of the $H$-catenoid but a subset of the horocylinder that projects onto a horocycle orthogonal to the $y$-axis at $(0,1)$. The horocylinder can be parameterized as $\alpha\times\R$ with $\alpha:(-\pi,\pi )\to \h^2$ given by
	\[\alpha(s)=\tfrac{1}{2}(\sin(s),1+\cos(s)).\]  We deduce that $x_0^n\to 0$ and $y_0^n\to 0$.
	
	We translate and rotate the surface $\Sigma_{\varphi^n}^-(a_1,\infty,f_1(a_1,\varphi^n))$ such that, in the half-space model, $(h_1^-)^n$ is contained in the vertical plane $\{x=0\}$ and $v_0^{n-}$ and $v_2^{n-}$ are contained in the horizontal plane $\h^2\times\{0\}$. In the limit, we have that the projection of $(v_0^-)^\infty$ and $(v_2^-)^\infty$ intersect twice the geodesic $\{x=0\}\subset\h^2$, see Figure~\ref{Fig-Lim-Inf} ((B)  - up left). By the continuity of conjugation (see~\cite[Proposition 3.3]{CMT}), for large $n$, the curves $\pi(v_0^{n-})$ and $\pi(v_2^{n-})$ also intersect twice the $y$-axis, see Figure~\ref{Fig-Lim-Inf}. Moreover, the curve $\pi((h_2^-)^n)$ is a non compact curve contained in a geodesic that cannot intersect the $y$-axis. Then we rotate the surface $\Sigma_{\varphi^n}^-(a_1,\infty,f_1(a_1,\varphi^n))$ until $(h_2^-)^n$ is contained in the vertical plane $\{x=0\}$ (the setting of the second period problem) and we have that the projections of  $v_0^{n-}$ and $v_2^{n-}$ intersect twice the vertical plane containing the curve $(h_1^-)^n$, see~Figure~\ref{Fig-Lim-Inf} ((B) - up right). In particular, we have that $\psi_0^n\in(2\pi,3\pi)$, that is, $\sin(\psi_0^n)>0$. Moreover, the curve $\gamma^n$ does not intersect the $y$-axis. We deduce from  Equation~\eqref{eq:gamma0-2} that $\mathcal P_2^1(a_1,\varphi^n,f(a_1,\varphi^n))<-1$ for $n$ large enough since $\sin(\psi_0^n)>0$.	
\end{proof}

\begin{theorem}\label{th-k-nodoids}
	For each $k\geq 2$, there exists properly Alexandrov-embedded  $H$-surfaces with $0<H\leq\frac{1}{2}$ in $\h^2\times\R$ with genus $1$ and $k$ ends. These $H$-surfaces have dihedral symmetry with respect to $k$ vertical planes and they are symmetric with respect to a horizontal plane. Moreover for $0<H<\tfrac 1 2 $, each of their ends is asymptotic to (and contained in the convex part of)  a vertical $H$-cylinder.
\end{theorem}
\begin{proof}
	Assume that $0<H<\tfrac 1 2$ and fix $0<\varphi<\frac{\pi}{2}$. By Lemma~\ref{lemma-P2-2}, 
	$\mathcal P_2^1(a_1,\varphi,f_1(a_1,\varphi))$ tends to $\cos(\varphi)$ when $a_1\to0$. If $\mathcal P_2^1(a_1,\varphi,f_1(a_1,\varphi))$ becomes greater than $1$ as $a_1\to a_{\mathrm{max}}(\varphi)$, then by the continuity of $\mathcal P_2^1$ we have that there exists $a_\varphi$ such that $ \mathcal P_2^1(a_\varphi,\varphi,f_1(a_\varphi,\varphi))=\cos(\frac{m\pi}{k})$ for all $k\geq 3$ and $m<k$ with $\text{gcd}(m,k)=1$ satisfying $\cos(\varphi)<\cos(\frac{m\pi}{k})$. On the other hand, if $\mathcal P_2^1(a_1,\varphi,f_1(a,\varphi))$ gets smaller than $-1$ as $a\to a_{\mathrm{max}}(\varphi)$, then there exists $a_\varphi$ such that $ \mathcal P_2^1(a_\varphi,\varphi,f_1(a_\varphi,\varphi))=\cos(\frac{m\pi}{k})$ for all $k\geq 2$ and $m<k$ with $\text{gcd}(m,k)=1$ satisfying $\cos(\frac{m\pi}{k})<\cos(\varphi)$.  We know that if $\varphi$ is close to $0$ and $a$ is close to $a_{\mathrm{max}}(\varphi)$, then $\mathcal{P}_2^1(a_1,\varphi,f_1(a_1,\varphi))>1$ and for values of $\varphi$ close to $\frac{\pi}{2}$ and $a_1$ close to $a_{\mathrm{max}}(\varphi)$ we have that $\mathcal{P}_2^1(a_1,\varphi,f_1(a_1,\varphi))<-1$. Then, by varying $\varphi\in(0,\frac{\pi}{2})$
	we find values of $\varphi$ and $a_\varphi$ such that $\mathcal{P}_2^1(a_\varphi,\varphi,f_1(a_\varphi,\varphi))=\cos(\frac{m\pi}{k})$ for all $m<k$ and $\text{gcd}(m,k)=1$. 
	
	Therefore, $\Sigma^-_\varphi:=\Sigma^-_\varphi(a_\varphi,\infty,f_1(a_\varphi,\varphi))$ solves the two period problems, and then after successive reflections over the vertical planes and the horizontal plane of symmetry, we obtain a complete $H$-surface with genus $1$ and $k$ ends asymptotic to vertical cylinders from the convex side. 
	
	Now assume that $H=\frac{1}{2}$ and consider the foliation $\{\alpha_\phi:[0,1]\to\Omega\}_{\phi\in(0,\frac \pi 2)}$ defined in Equation~\eqref{eq:foliation}. Set $k\geq 2$ and $m<k$ with $\text{gcd}(m,k)=1$ and choose $\phi$ such that $ \cos(\frac{m\pi}{k})<\cos(\phi)$. By Lemma~\ref{lemma-P2}, we have that $\mathcal P_2^1(\alpha_\phi(0),f_1(\alpha_{\phi}(0)))=\cos(\phi)$ and $\mathcal{P}_2^1(\alpha_\phi(t),f_1(\alpha_{\phi}(t)))<-1$ for $t$ close enough to $1$. We deduce that there exist $a(\phi)$ and $\varphi(\phi)$ such that $\mathcal{P}_2^1(a(\phi),\varphi(\phi),f_1(a(\phi),\varphi(\phi)))=\cos(\frac{m\pi}{k})$. Then the surface \[\Sigma_\phi^-:=\Sigma_{\varphi(\phi)}^-(a(\phi),\infty,f(a(\phi),\varphi(\phi)))\] solves both periods problems, and we obtain a complete $H$-surface with genus one and $k$ ends after successive reflections over the vertical planes and the horizontal plane of symmetry.
\end{proof}

\begin{proposition}
For $H=\tfrac 12$, the $(H,k)$-nodoids with genus one and $k\geq2$ ends are never embedded.
\end{proposition}
\begin{proof}
	We will prove that the ideal extreme  $\pi(v_2^-)$ is $(0,0)$.  As the curve $\gamma$ intersects the $y$-axis when $|\mathcal P_2^1|<1$, this means that $\pi(v_0^-)$ or $\pi(v_2^-)$ must cross $\gamma$ and then $\Sigma_\phi^-$ is not embedded after the reflection over the curve $h_1^-$.  

We use similar ideas to those in the proof of the embeddedness of the $\frac{1}{2}$-noids with genus one, see Proposition~\ref{Theorem:embebimiento}.
Let consider $\widetilde\Sigma_\phi^n:=\widetilde\Sigma_{\varphi(\phi)}(a(\phi),n,f_1(a(\phi),\phi))$ the sequence of minimal graphs over $\widetilde\Delta(n,a_1,\varphi(\phi))$ converging to $\widetilde\Sigma_\phi$ and its respective conjugate surfaces (after the reflection over $h_2$) $(\Sigma_\phi^-)^n$ converging to $\Sigma_\phi^-$. On the one hand, let $\widetilde v_1^n\subset\partial\widetilde\Sigma_\phi^n$ and $\widetilde v_2^n\subset\partial\widetilde\Sigma_\phi^n$ be the vertical geodesics projecting onto $p_1^n$ and $p_2^n$ respectively, and let $v_1^{n-}$ and $v_2^{n-}$ be their conjugate curves contained in  horizontal planes.  Let $k_g^n=1-(\theta_1^n)'$ be the curvature of $v_1^{n-}$ with respect to the normal that points to the exterior of the domain $\Delta^n$ where $(\Sigma_\phi^-)^n$ is projecting. We know that $k_g^n$ approaches $1$ as $n\to\infty$.  On the other hand, the second coordinate  of $(h_2^-)^n$ diverges since we have shown that  $\pi(h_2)$ is not compact. We have that $\pi(v_1^{n-})$ approaches a half of a horocylinder with arbitrary large euclidean radius with the ideal extreme in $(0,0)$ that contains the endpoint of $\pi((h_1^n)^-)$ in the line $\{x=0\}$. That proves that the ideal extreme of $\pi(v_2^{n-})$ converges to $(0,0)$ and in particular $\pi(v_1^{n-})$ approaches the asymptotic boundary $\{y=0\}\cup\{+\infty\}$ as $n\to \infty$.	
\end{proof}

\begin{remark}
For $H<\frac 1 2$	we can prove that we have examples of genus $1$ when the second-period function is negative. However, it seems complicated to decide if the sign of $\sin (\psi_0)$ or the sign of $x_0$ are positive or negative in any case. This produces different kinds of $H$-surfaces, depending on these signs as we sketch out in Figure~\ref{Fig-Casos}.

\end{remark}

\begin{figure}[htb]
	\begin{center}
		\includegraphics[height=11cm]{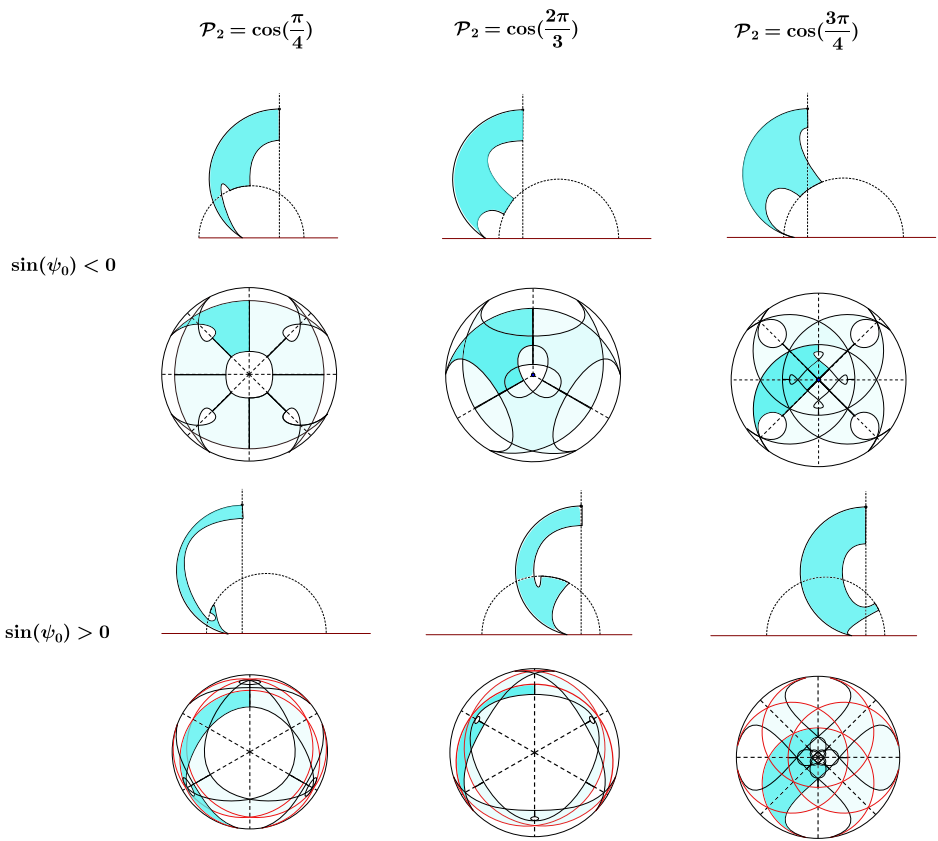}
	\end{center}
	\caption{The projections of different cases of the fundamental piece with $\mathcal P_2^1=\cos(\frac \pi 4)$, $\mathcal P_2^1=\cos(\frac{2\pi}{3})$ and $\mathcal P_2^1=\cos(\frac{3\pi}{4})$ in the half-space model and in the disk model of $\h^2$. Up, the cases with $\sin(\psi_0)<0$ and, down, the cases with $\sin(\psi_0)>0$.}
	\label{Fig-Casos}
\end{figure}

\begin{remark}
	In the case of $k=2$, that is, the second period function vanishes, we have two possibilities depending on the sign of $\sin(\psi_0)$, see Figure~\ref{Fig-2-nodoides}. We expect that these examples with $2$ ends and genus $1$ are never embedded. At least, they should not be embedded for $H$ near $0$, since there are not examples for $H=0$ by the uniqueness of the horizontal catenoid proved in~\cite{HNST}. In that case  for  $H$ close to $0$,  our examples with $2$ ends should be near to a vertical plane.  
\end{remark}

\begin{figure}[htb]
	\begin{center}
		\includegraphics[height=6cm]{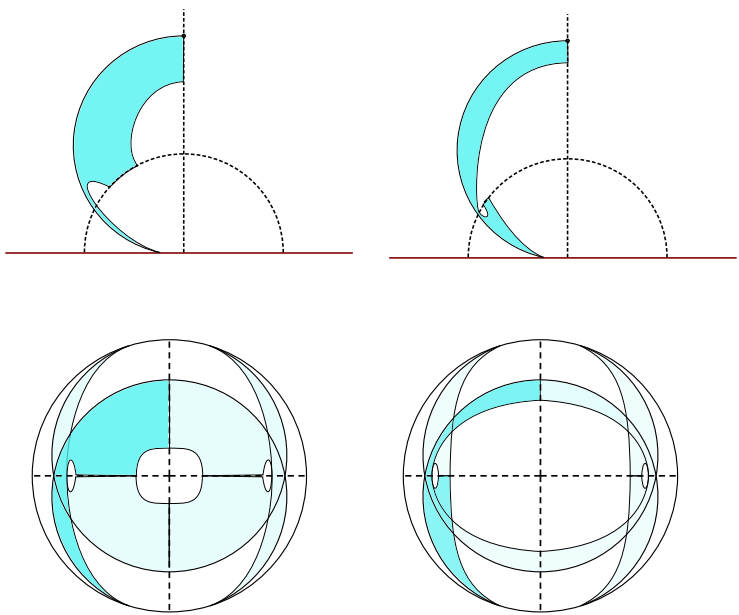}
	\end{center}
	\caption{The projections  of the $(H,2)$-nodoids with genus $1$. On the left, we have the case with $\sin(\psi_0)<0$ and, on the right, the case with $\sin(\psi_0)>0$.}
	\label{Fig-2-nodoides}
\end{figure}

\subsection*{Acknowledgments} The authors would like to express their gratitude to José Miguel Manzano for his valuable comments during the preparation of this paper. This research is supported by MCIN/AEI project PID-2019-111531GA-I00. The first author is also partially supported by the FEDER/ANDALUCÍA P18-FR-4049 and by the MCIN/AEI project PID-2020-117868GB-I00. The second author is also supported by a PhD grant funded by University of Jaén and by a FEDER-UJA grant (Ref. 1380860).

\end{document}